%% file: main.tex
\documentclass[10pt,twocolumn,twoside]{IEEEtran}

\usepackage{xcolor}

\usepackage{amsmath,amssymb,array,arydshln}
\usepackage{amsthm}
\usepackage{bm}
\allowdisplaybreaks

\usepackage{graphicx}

\usepackage{mathtools}
\usepackage{algorithm}
\usepackage{algpseudocode}
\usepackage[caption=false, font=footnotesize]{subfig}
\usepackage{subfloat}
\usepackage{multirow}
\usepackage{titlesec}
\usepackage{hyperref}

\usepackage[opt]{icpslab}

\input{deflatex.tex}

\newtheorem{theorem}{Theorem}

\newtheorem{lemma}{Lemma}
\newtheorem{scenario}{Scenario}

\newtheorem{proposition}{Proposition}
\newtheorem{remark}{Remark}

\ifdefined\SubmittedPaper
\newcommand{\changed}[1]{\textcolor{blue}{#1}}
\else
\newcommand{\changed}[1]{#1}
\fi

\ifdefined\SubmittedPaper
\titlespacing*{\section}{0pt}{0.4em}{0.3em}
\titlespacing*{\subsection}{0pt}{0.3em}{0.2em}
\titlespacing*{\subsubsection}{0pt}{0.2em}{0.1em}
\fi

\def\BibTeX{{\rm B\kern-.05em{\sc i\kern-.025em b}\kern-.08em
    T\kern-.1667em\lower.7ex\hbox{E}\kern-.125emX}}
\begin{document}

\title{Communication-efficient ADMM over Hierarchical Networks}

\author{Binh Nguyen, Shuangqing Wei and Truong X. Nghiem
\thanks{This material is based upon work supported by the National Science Foundation under Awards No. 2449927 and 231711.}%
\thanks{Corresponding author: Truong X. Nghiem (truong.nghiem@ucf.edu).}%
\thanks{Binh Nguyen and Truong X. Nghiem are with Department of Electrical and Computer Engineering,
College of Engineering and Computer Science, 
University of Central Florida, Orlando, FL 32816, USA (e-mail: thanhbinh.nguyen@ucf.edu; truong.nghiem@ucf.edu).}%
\thanks{Shuangqing Wei is with Division of Electrical and Computer Engineering, School of Electrical Engineering and Computer Science, Louisiana State University, Baton Rouge, LA 70803, USA (e-mail: swei@lsu.edu).}%
}

\maketitle

\begin{abstract}
This paper develops a novel distributed optimization algorithm based on the Alternating Direction Method of Multipliers (ADMM) to solve hierarchical optimization problems over tree-structured networks, termed hierarchical ADMM (hADMM), with a particular focus on enhancing communication efficiency across the network. 
By rearranging the augmented Lagrangian to establish a query-response communication mechanism between nodes that explicitly exploits the hierarchical tree structure, the proposed algorithm significantly reduces communication costs compared to existing ADMM-based methods for hierarchical optimization. 
Furthermore, hADMM guarantees asymptotic convergence under convexity assumptions. 
We also present a convergence rate analysis based on linear matrix inequalities to characterize the maximum theoretically achievable convergence rates across different network topologies, showing that the proposed hADMM attains linear convergence under mild conditions.
Three numerical experiments demonstrate that hADMM is compatible with arbitrary tree network structures and outperforms existing approaches in terms of communication efficiency.
\end{abstract}


\section{Introduction}
\label{sec:introduction}

\changed{The alternating direction method of multipliers (ADMM) is a widely used optimization method for solving large-scale,} 
distributed optimization problems characterized by structured objective functions and coupled constraints \cite{boyd2011distributed}. 
ADMM integrates dual decomposition strategies with augmented Lagrangian methods, effectively breaking down complex optimization tasks 
into manageable subproblems that can be solved independently and concurrently. The method iteratively updates primal and dual variables, 
allowing distributed agents to solve local optimization tasks and subsequently coordinate solutions through dual variables. 
Due to its efficient convergence properties, ease of parallelization, and scalability, ADMM has seen wide adoption across numerous fields, 
including signal processing \cite{chanPlugandPlayADMMImage2017}, machine learning \cite{wang2019global}, smart grid management \cite{erseghe2014distributed}, and robotics \cite{shorinwa2024distributed}.

Distributed optimization over hierarchical networks (or simply hierarchical optimization) has attracted increasing attention due to its 
relevance in large-scale systems where agents are naturally organized in multi-level topologies. 
Such structures commonly arise in many applications
\cite{biswas2022decentralized, nguyen2023distributed, reddy2024supply, noah2024distributed}.
For example, in an electrical power distribution grid, each load, such as a building, is a leaf node with internal variables, dynamics, constraints, and total power demand determining its operating cost.
Substations serving clusters of loads include transformers and equipment, %and have
with their own dynamics, constraints, and aggregate demand based on connected loads.
Each substation and its associated loads form a subtree.
A central grid coordinator at the tree's root maintains grid stability and safety while providing services such as demand response \cite{biswas2022decentralized, tsaiCommunicationEfficientDistributedDemand2017a,sianoDemandResponseSmart2014a}, by coordinating substations and loads.
In smart grids, two-way communication between nodes and their parents supports real-time demand adjustment and coordination, thus minimizing costs, ensuring reliability, and enhancing grid service performance.
This setup results in a hierarchical optimization problem solvable in a distributed fashion.

Unlike flat peer-to-peer networks, hierarchical networks impose parent–child relationships and asymmetric information flows, which introduce additional challenges for distributed optimization, particularly in terms of coordination, scalability, and communication efficiency.
Classical distributed optimization methods designed for fully connected or general graphs, \eg \cite{nedic2009distributed, nedic2010constrained, motaDADMMCommunicationEfficient2013, makhdoumi2017convergence, yang2022survey}, often \emph{fail to fully exploit the structural properties}, leading to unnecessary communication overhead or reliance on centralized coordination.
In particular, general hierarchical optimization problems can be solved by two simple ADMM-based methods.
The first method, which we call the nested ADMM (nADMM) and was presented in \cite{khakiHierarchicalADMMBased2018,zhangCooperativeLoadScheduling2020a}, solves a hierarchical sharing problem by applying classical ADMMs recursively to each layer of the tree, where each subtree is regarded by the ADMM applied at its parent as a single node for optimization purposes.
The algorithm therefore involves multiple nested ADMM loops, one for each level of the tree.
The second method, called the flattened ADMM (fADMM), flattens a tree network into a star network by establishing relay communication between each node and the root node. 
Conventional ADMM is then applied to the resulting star network, treating the original hierarchical structure as an equivalent star network.
Both methods incur significant communication overhead, as we will show later in the experiments.

\changed{Recent communication-efficient ADMM methods reduce communication through additional local mixing, selective and quantized messages, or application-specific message elimination.
The method in \cite{huangCommunicationEfficientDecentralized2026} addresses decentralized composite consensus over general peer-to-peer graphs and uses a symmetric ADMM with two consecutive neighbor exchanges per iteration to incorporate information beyond immediate neighbors.
Authors in \cite{duarteCommunicationefficientADMM2026a} considers an optimal sharing problemand reduces transmitted bits by predicting local proximal responses with Gaussian process regression, selectively querying agents, and adaptively quantizing their messages.
The method in \cite{iqbalCommunicationEfficientDistributed2026} specializes ADMM to distributed Kalman filtering over an undirected estimator network and reduces message payload by eliminating the exchange of dual variables between neighboring estimators.
However, none of these approaches addresses general multilevel (hierarchical) optimization on a tree with affine constraints coupling parent and child decision variables.
}

Due to the above challenges, there is a need for \emph{structure-aware distributed optimization} frameworks that explicitly leverage the hierarchical topology to enable scalable, communication-efficient, and fully distributed computation.
This paper presents a novel ADMM-based algorithm, called the hierarchical ADMM (hADMM), that effectively exploits the hierarchical structure to %drastically
improve communication efficiency.
Our \textbf{main contributions} are summarized as follows:
\begin{itemize}
    \item We propose a novel hADMM algorithm for efficiently solving general hierarchical optimization problems over arbitrary tree-structured networks, achieving substantial reductions in network communication costs in comparison with the fADMM and the nADMM.
    \changed{Unlike nADMM and fADMM, hADMM follows the original parent-child edges: query information is passed top-down, local primal responses are returned bottom-up, and each node communicates only with its parent and children.}
    \item 
    We provide rigorous theoretical analyses establishing the convergence and convergence rate of the proposed method. 
    In particular, the hADMM is guaranteed to converge under general convexity assumptions, even in the presence of nonsmooth convex cost functions.
    \item We demonstrate the superior communication efficiency of the hADMM compared with the fADMM and the nADMM  and validate the theoretical results through comprehensive and systematic numerical experiments. 
\end{itemize}

\noindent\textbf{Notation:}
Let $\RR$, $\RR_{>0}$, $\bbN$, and $\bbN_{>0}$ be the sets of real numbers, positive real numbers, integers, and positive integers, respectively.
\changed{
For a finite-valued convex function $f: \bbR^n \rightarrow \bbR$, $\partial f(x)$ denotes the subdifferential of $f$ at $x$.
In addition, $
\partial f(x)\neq\varnothing \; \forall x \in \bbR^n$.
For a real sequence $\{a_k\}$, $\bar a$ is a cluster point of  $\{a_k\}$  if, for given  $\epsilon \in \bbR_{>0}$  and given  $N>0$,  $\exists n \in \bbN_{>0}$  such that  $\Vert a_n - \bar a\Vert_2 \leq \epsilon$.
}
For $D \in \bbN_{>0}$, define $\bbN_D = \{1,\dots, D\}$.
In symmetric block matrices, the asterisk $(*)$ is used as an ellipsis for terms induced by symmetry.
For integers $n, m\in \bbN_{>0}$, $I_n \in \RR^{n\times n}$ denotes the identity matrix of size $n$, $0_n$ is the $n$-dimensional zero vector, and $0_{n\times m}$ is the $n \times m$ zero matrix.
For an ordered index list $\calI = \{i_1, i_2,\dots, i_n \}$, $i_k \in \bbN_{>0}$, and real matrices $M_{i_1}, M_{i_2}, \dots, M_{i_n}$ in appropriate dimensions, let $\{M_i\}_{i \in \calI}$ be the ordered list of the matrices, and denote
\begin{align*}
  &[M_{i}]_{i \in \calI} =
    \begin{bmatrix}
      M_{i_1}^\top, M_{i_2}^\top, \dots, M_{i_n}^\top
    \end{bmatrix}^\top,
  \\
  &[M_{i}]_{i \in \calI}^{\diag} = \diag(M_{i_1}, \dots, M_{i_n}) =
    \begin{bmatrix}
      M_{i_1} \\&\ddots\\& &M_{i_n}
    \end{bmatrix}
    \text.
\end{align*}
For a set $\calS$, $|\calS|$ is the number of elements in $\calS$. 
For a real matrix $M$, $\Vert M \Vert_2$ denotes the matrix 2-norm (spectral norm) of $M$, which is %equal to
the largest singular value of $M$.
For two square \changed{symmetric} matrices $A$ and $B$ of the same dimensions, $A$ $\succ(\succeq)$ $B$ implies that $A-B$ is a positive (semi)definite matrix.
For matrices $U_1, \dots, U_{n-1}$, $D_1, \dots, D_n$, and $L_1,\dots, L_{n-1}$ in appropriate dimensions, denote the tridiagonal block matrix 
\begin{align*}
  \tridiag\left(\{U_i\}_{i \in \bbN_{n-1}},  \{D_i\}_{i \in \bbN_n}, \{L_i\}_{i \in \bbN_{n-1}}\right) = 
  \\
  \begin{bmatrix}
    D_1 &U_1 & & \\
    L_1 &D_2 &\ddots & \\
      &\ddots &\ddots &U_{n-1} \\
      & &L_{n-1} &D_n
  \end{bmatrix}\text.
\end{align*}

\section{Problem Formulation and Related Literature}
\label{sec:prob-formulation}

This section formulates a hierarchical optimization problem defined over a tree network.
We then present the nested and flattened ADMM schemes, which serve as baseline approaches for the comparative evaluation of our proposed methods.

\subsection{Hierarchical optimization over tree networks}
\label{sec:tree-opt}

\begin{figure}[!tb] 
  \centering
  \includegraphics[width = \linewidth]{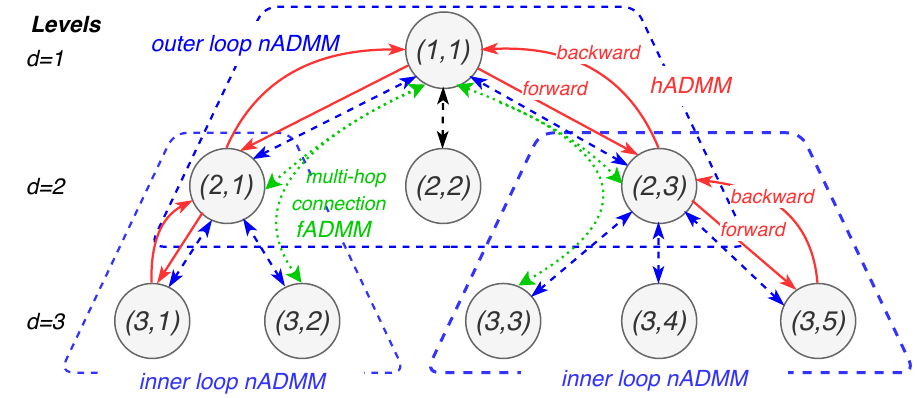}
  \caption{An example of bidirectional tree networks: the root node, parent nodes, and leaf nodes are red, blue, and green, respectively.  \changed{It also illustrates the nested ADMM scheme (blue dashed line), flatten ADMM scheme (green dotted line), hierarchical ADMM scheme (red solid line) presented in Section~\ref{sec:nADMM}, \ref{sec:flattenADMM}, and \ref{sec:hADMM}}, respectively.}
  \label{fig:treeNet}
\end{figure}

Let us consider a bidirectional tree network $\calG$ with $D \in \bbN_{>0}$  levels (layers).
A pair $(d,i)$ stands for the $i$-th node at level $d$, where $d=1$ is the root level and $(1,1)$ is the one and only root node.
A parent node is a node at level $d \in \{1,\dots,D-1\}$ which is connected to at least a node at the next level $d+1$, called a child node.
A node that is not a parent node is called a leaf node.
Each non-root node has exactly one parent node and may have multiple child nodes.
A node at level $d$ can communicate with only its parent node and its child nodes through their direct connections, hence it is $(d-1)$ hops from the root node in terms of communication.

Let $\calV \subset \bbN_{>0}^2$ be the set of all nodes.
For node $(d,i) \in \calV$, denote $\calC_{d,i} = \{j \in \bbN_{>0} | (d+1,j) \in \calV \text{~is a child of~} (d,i) \}$ as the set of its children.
Let $\calP_d = \{i \in \bbN_{>0}| (d,i) \in \calV, \calC_{d,i} \neq \emptyset \}$ be the set of all parent nodes at level $d$ and $\calL_d= \{i \in \bbN_{>0}| (d,i) \in \calV, \calC_{d,i} = \emptyset \}$ be the set of all leaf nodes at level $d$.
Note that $\calP_d \cup \calL_d$ contains the indices of all nodes at level $d$.
For example, consider the tree network with three levels ($D=3$) in Fig.~\ref{fig:treeNet}.
The first layer includes the root node denoted by $(1,1)$ with $\calP_1 = \{1\}$, $\calC_{1,1} = \{1,2,3\}$, and $\calL_1 = \emptyset$.
The second layer includes three nodes $(2,1), (2,2)$, and $(2,3)$, with $\calP_2 = \{1,3\}$, $\calL_2 = \{2\}$, $\calC_{2,1} = \{1,2\}$, $\calC_{2,2} = \emptyset$, and $\calC_{2,3} = \{3,4,5\}$.
The third layer has only leaf nodes, with $\calP_3 = \emptyset$ and $\calL_3 = \{1,2,3,4,5\}$.

Each node $(d,i)$ has a vector of decision variables $x_{d,i} \in \RR^{n_{d,i}}$ 
and a cost function $f_{d,i}(x_{d,i})$.
This paper considers the following \textit{hierarchical optimization} problem over $\calG$:
\begin{subequations}
\label{eq:prob-tree-opt}
\begin{align}
    \minimize \; &\changed{J =} \sum_{d=1}^D \bigg( \sum_{i\in\calP_d}  f_{d,i}(x_{d,i}) \!+\!   
    {\sum_{j \in \calL_d} f_{d,j}(x_{d,j})}  \bigg),
    \label{eq:prob-tree-opt:obj}
    \\
    \text{s.t.} \; & A_{d,ij} x_{d,i} + B_{d,ij} x_{d+1,j} = c_{d,ij}\text,
    \label{eq:tree-cstr}
\end{align}
\end{subequations}
where \eqref{eq:tree-cstr} describes the relationship between node $(d+1,j)$ and its parent $(d,i)$ and holds for all $d \in \bbN_{D-1}$, %$= 1,\dots,D-1$, 
$i \in \calP_d$, and $j\in \calC_{d,i}$,
with matrices $A_{d,ij} \in \RR^{s_{d,ij}\times n_{d,i}}$, $B_{d,ij} \in \RR^{s_{d,ij}\times n_{d+1,j}}$, and 
vector $c_{d,ij} \in \RR^{s_{d,ij}}$.
It is assumed that $f_{d,i}: %(x_{d,i}): 
\changed{\bbR^{n_{d,i}}\rightarrow \bbR}$ are finite-valued convex functions for all $(d,i) \in \calV$, thus the optimization problem \eqref{eq:prob-tree-opt} is convex.

\subsection{Disjoint problems over tree networks}
\label{sec:app}

The \textit{disjoint problem} is a special case of the hierarchical optimization problem \eqref{eq:prob-tree-opt}, that is often encountered in practice.
In this problem, the decision vector of a parent node at level $d$ is the concatenation of the decision  vectors of its children,  \changed{in the increasing order of the  indices in $\calC_{d,i}$}, that is, $x_{d,i} = [x_{d+1,j}]_{j\in \calC_{d,i}}$.
The resulting problem
\begin{align}
    \changed{\minimize \; J} \; \text{s.t.} \; x_{d,i} = [x_{d+1,j}]_{j\in \calC_{d,i}}, d \in \bbN_{D-1}, i \in \calP_d
    \label{prob:disjoint}
\end{align}
\changed{minimizes the same objective $J$ as in \eqref{eq:prob-tree-opt:obj}  but with the coupling constraint \eqref{eq:tree-cstr} specialized to the concatenation constraint. 
It} can be formulated as the general problem \eqref{eq:prob-tree-opt} with
$n_{d,i} = \sum_{j\in \calC_{d,i}} n_{d+1,j}$, 
$c_{d,ij} = 0$,
$B_{d,ij} = -I_{n_{d+1,j}}$, and
$A_{d,ij} = \begin{bmatrix} 0 &\dots &I_{n_{d+1,j}} &\dots &0\end{bmatrix} \in \RR^{n_{d+1,j} \times n_{d,i}}$.
\changed{Here, the block columns of $A_{d,ij}$ are ordered consistently with the order of the child nodes in $\calC_{d,i}$, the block corresponding to %the child node 
$j \in \calC_{d,i}$ is the identity matrix $I_{n_{d+1,j}}$,
and %all the other blocks
the rest are zeros}. % matrices.
Thus, $A_{d,ij}$ effectively selects the variables $x_{d+1,j}$ corresponding to the child node $j$ in the vector $x_{d,i}$.
Note that $\left[A_{d,ij}\right]_{j\in \calC_{d,i}} = I_{n_{d,i}}$.

The rest of this section describes two standard algorithms for solving the general hierarchical optimization problem \eqref{eq:prob-tree-opt}.

\subsection{Nested ADMM (nADMM) algorithm}
\label{sec:nADMM}

The standard ADMM can be adapted to solve \eqref{eq:prob-tree-opt}.
Specifically, ADMM solves the two-block coupled problem \cite{boyd2011distributed}
\begin{align}
    \minimize_{u,v} \; g(u) + h(v) \;\text{s.t.}\; Au + Bv = c \text,
    \label{eq:prob-basicADMM}
\end{align}
with convex functions $g$ and $h$, by using the following uncoupled steps, for $\rho > 0$,
\begin{align} \label{eq:basicADMM_iter}
    % \!\!\!\!\!
\begin{cases*}
    u^{k+1} = \argmin_u g(u) \!+\! \frac{\rho}{2}\Vert Au \!+\! B v^{k} \!-\! c \!+\! \frac{1}{\rho} y^k\Vert_2^2,
    \\
    v^{k+1} = \argmin_v h(v) \!+\! \frac{\rho}{2}\Vert Au^{k+1} \!+\! B v \!-\! c \!+\! \frac{1}{\rho} y^k \Vert_2^2,
    \\
    y^{k+1} = y^k + \rho(Au^{k+1} + B v^{k+1} - c)\text.
\end{cases*}
\end{align}
As shown in \cite{boyd2011distributed}, the iterations \eqref{eq:basicADMM_iter} ensure the stable sequence $(u^k, v^k, y^k)$, 
which converges to the optimal solution when $g$ and $h$ are closed proper convex functions, and $A$ or $B$ is a full rank matrix \cite{heAlternatingDirection2012}.
The authors in \cite{goldstein2014fast} show that the ADMM  \eqref{eq:basicADMM_iter} converges in the dual objective with linear rate, and they further propose an accelerated variant of \eqref{eq:basicADMM_iter}\changed{.} 
Other results on linear convergence of the ADMM are shown in \cite{deng2016global}.
A more general variant of \eqref{eq:prob-basicADMM}, called the multiple-block problem, 
\begin{align}
    \minimize_{u,v_1,\dots, v_N}~ g(u) + \sum_{i=1}^{N} h_i(v_i)
    \;\text{s.t.}\; Au + \sum_{i=1}^N B_i v_i = c\text,
    \label{eq:prob-mulblock}
\end{align}
has been solved by Jacobi-type ADMM \cite{linGlobalLinear2015} and proximal Jacobi-type ADMM \cite{deng2017parallel} with linear convergence guarantee. % rate.
\changed{
These multi-block ADMM variants provide important convergence guarantees, but their single global coupling constraint naturally induces a star topology with a central coordinator. Stacking the variables and constraints into a large block-structured matrix does not yield a distributed implementation on the original tree, since the coordinator must still collect aggregated information while intermediate nodes relay messages to emulate the star topology, similar to the flattened ADMM algorithm presented in Section~\ref{sec:flattenADMM}.}

Both \eqref{eq:prob-basicADMM} and its generalization \eqref{eq:prob-mulblock} are formulated for star networks, %(with depth $D=2$),
where $g$ is the cost function of the coordinator and $h_i$ is the local cost function of the leaf node $i$.
For solving the hierarchical optimization \eqref{eq:prob-tree-opt}, the authors of \cite{khakiHierarchicalADMMBased2018,zhangCooperativeLoadScheduling2020a,etedadi2024hierarchical} apply standard ADMM \eqref{eq:basicADMM_iter} recursively to each layer of the tree structure.
We term this method \textit{nested ADMM} (nADMM). 
Other variants of %An other example of
nADMM can be found in \cite{mirzaeifard2025smoothing,azimi-abarghouyiHierarchicalFederated2025}, which address federated learning over tree networks.
In nADMM, each subtree is regarded by the ADMM applied at its parent node as a single node for optimization purposes.
The algorithm therefore involves multiple nested loops, one for each level of the tree.
In Fig.~\ref{fig:treeNet}, the root node $(1,1)$ coordinates the global optimization via an outer ADMM loop (red box). 
Each query to a depth-2 cluster node, such as $(2,1)$ or $(2,3)$, triggers an inner ADMM loop (blue box) to solve its subtree's local proximal minimization. 
Each inner loop completes multiple query-response rounds before returning a response to the root.

While the algorithms in \cite{khakiHierarchicalADMMBased2018} and \cite{zhangCooperativeLoadScheduling2020a} are designed only for the disjoint problem \eqref{prob:disjoint} over tree networks with depth $D = 3$, we generalize the nADMM for the hierarchical optimization problem \eqref{eq:prob-tree-opt} over a general tree network in 
\ifdefined\SubmittedPaper
\cite[Appendix~A]{hADMMfullpaper}.
\else
Appendix~\ref{app:nADMM}.
\fi

\subsection{Flattened ADMM (fADMM) algorithm}
\label{sec:flattenADMM}

In the flattened ADMM, the global objective function \eqref{eq:prob-tree-opt:obj} is flattened by rewriting it as a single sum of the local objective functions of all individual nodes, then the standard ADMM is applied where all nodes are considered equal, \ie each iteration involves solving the proximal minimization problem of 
every node in the tree.
A non-root parent node will relay the data exchanges between its children and its parent. 
From the computation perspective, the tree is flattened into a star structure.
From the communication perspective, however, the tree structure is retained as the data flow follows the tree.

\changed{The fADMM is used as a baseline that applies standard star-topology ADMM to the flattened hierarchical problem.
It is not an established hADMM algorithm, but a direct baseline that ignores the hierarchy in the optimization formulation.}

\subsection{Communication and computation of nADMM and fADMM}
\label{sec:comm-cost}

The fADMM incurs significant communication overhead as messages must be exchanged between the root and every node in the network along a multi-hop path between them, where all nodes in the path must relay messages in both directions.
Such a process is particularly costly in terms of communication when the non-root node resides at the deepest level of the tree \cite{zhangCooperativeLoadScheduling2020a}.
Moreover, the dependence of fADMM on multi-hop communication increases network protocol complexity and generates additional traffic on the underlying infrastructure, which can compromise its overall reliability. 
The overall computation cost of the fADMM is proportional to the network size (\ie the number of nodes) and the number of iterations.

In contrast to the fADMM, the nADMM only requires direct (single-hop) communication between a parent node and a child node.
However, this advantage comes at the cost of high computational complexity and does not reduce the overall communication cost.
Specifically, each iteration of an outer loop involves recursive computations in the inner loops, leading to an exponential growth in communication and computation costs as the network depth increases.
As a result, lower-level nodes perform more iterations than 
do upper-level nodes.
Consequently, the root node incurs the least communication and computation costs, while nodes at the deepest level have the highest communication and computation burden.
Overall, the total communication and computation costs of the nADMM are higher than those of the fADMM.

\changed{The example in Fig.~\ref{fig:treeNet} illustrates these costs.
Let $T_{2,1}$ and $T_{2,3}$ denote the numbers of inner ADMM iterations at nodes $(2,1)$ and $(2,3)$ per outer nADMM iteration.
Root-level query-response exchanges with the three level-2 nodes require $6$ transmissions.
The inner solves at nodes $(2,1)$ and $(2,3)$ require, respectively, $4T_{2,1}$ and $6T_{2,3}$ additional parent-child transmissions and $3T_{2,1}$ and $4T_{2,3}$ local subproblem solves.
Thus, each outer nADMM iteration requires at least $6+4T_{2,1}+6T_{2,3}$ transmissions and $2+3T_{2,1}+4T_{2,3}$ subproblem solves, where $2$ accounts for the root and leaf node $(2,2)$.
By contrast, each fADMM iteration requires $9$ local subproblem solves and, due to the multi-hop communication between the root and the lower-level nodes, $2 \times (3+2\times 5)=26$ transmissions, including $2 \times (2+3)=10$ transmissions relayed by nodes $(2,1)$ and $(2,3)$.
Hence, nADMM incurs nested-iteration costs, whereas fADMM incurs multi-hop relay traffic.}

\section{Hierarchical ADMM (hADMM) Algorithm}
\label{sec:hADMM}

We propose an algorithm, called \emph{hierarchical ADMM} (hADMM), for solving the hierarchical optimization problem \eqref{eq:prob-tree-opt}, where computation and communication naturally follow the tree structure of the network, resulting in computation and communication savings compared to the nADMM and fADMM algorithms.
These savings will be demonstrated in the numerical examples in Section~\ref{sec:examples}.
To improve readability of the mathematical derivation afterward, selected important notations are summarized in Table~\ref{tab:notations}.

\begin{table}[!t]
  \caption{Summary of Important Notations}
  \label{tab:notations}
  \centering
  \begin{tabular}{ll}
    $D$ & Depth of the tree network $\calG$\\
    $(d,i)$ & $i$-th node at $d$-th level of $\calG$ \\
    $\calC_{d,i}$ & Index set of all children of node $(d,i)$\\
    $\calP_d$ & Index set of all parent nodes at level $d$\\
    $\calL_d$ & Index set of all leaf nodes at level $d$\\
    $x_{d,i}$ & Decision variable vector of node $(d,i)$ \\
    $f_{d,i}$ & Cost function of node $(d,i)$ \\
    $x_d$ & Concatenation of all node variables at level $d$\\ 
    $s_{d,ij}$ &Number of scalar constraints between node $(d,i)$ 
    \\
    &and its child $(d+1, j)$
    \\
    %\hline
    \end{tabular}
\end{table}

\subsection{Augmented Lagrangian function}
\label{sec:lagrangian}

This section establishes an augmented Lagrangian function 
for the problem \eqref{eq:prob-tree-opt}.
We first define some useful notations.
The constraints \eqref{eq:tree-cstr} for $j\in \calC_{d,i}$ can be combined into an aggregate constraint between the parent node $(d,i)$ and all its child nodes $(d+1,j)$ as
\(A_{d,i} x_{d,i} + B_{d,i}[x_{d+1,j}]_{j\in \calC_{d,i}} = c_{d,i}\), 
where
$ A_{d,i} =  [A_{d,ij}]_{j \in \calC_{d,i}}$, 
$B_{d,i} =  [B_{d,ij}]_{j \in \calC_{d,i}}^{\diag}$, 
and $c_{d,i} = \big[ c_{d,ij}\big]_{j \in \calC_{d,i}}$. 
Let us define the vector of all node variables at level $d$ as
\(x_d =  \left[ [x_{d,i}]_{i \in \calP_d}^\top, [x_{d,i}]_{i\in \calL_d}^\top \right]^\top \in \RR^{n_d}\) 
and the sum of all local cost functions at level $d$ as
$f_d(x_d) = \sum_{i\in\calP_d \cup \calL_d}  f_{d,i}(x_{d,i})$.
The hierarchical optimization problem \eqref{eq:prob-tree-opt} can be expressed as the following multiple-layer optimization problem:
\begin{subequations}
    \label{prob:chain_mul_block_opt}
\begin{align} 
    \minimize_{x_1, \dots, x_D}\; &J = \sum_{d=1}^D f_d(x_d),
    \\
    \text{s.t.}\; & A_d x_d + B_d x_{d+1} = c_d, \; \forall d \in \bbN_{D-1}\text,
    \label{prob:chain_cstr}
\end{align}
\end{subequations}
where
$A_d =  \big[[A_{d,i}]_{i\in \calP_d}^{\diag},~ 0  \big] \in \bbR^{s_d \times n_d}$ with $s_d = \sum_{i\in\calP_d} \sum_{j\in \calC_{d,i}} s_{d,ij}$,
$B_d =  [B_{d,i}]_{i\in \calP_d}^{\diag} M_d$, 
$c_d = \big[ c_{d,i}\big]_{i\in \calP_d}$,
and $M_d$ is a permutation matrix for extracting the variables of all the child nodes at level $d+1$, \ie
$[[x_{d+1,j}]_{j\in \calC_{d,i}}]_{i\in \calP_d} = M_d x_{d+1}$.

We establish the following augmented Lagrangian for optimization problem \eqref{prob:chain_mul_block_opt} with penalty parameter $\rho > 0$:
\begin{multline}
  \label{def:Lagrangian}
  \changed{L(x_1,\dots,x_D; y_1,\dots,y_{D-1})}
  = \\
  f_D(x_D)
  + \sum_{d=1}^{D-1} \Big( f_d(x_d) + y_d^\top(A_dx_d + B_d x_{d+1} \!-\! c_d) \\
  + \frac{\rho}{2} \Vert A_dx_d + B_d x_{d+1} - c_d \Vert_2^2 \Big)\text,
\end{multline}
where
$y_d = [y_{d,i}]_{i\in \calP_d} \in \RR^{s_d}$,
$y_{d,i} = [y_{d,ij}]_{j\in \calC_{d,i}}$, and
$y_{d,ij}$ is the dual variable associated with the constraint \eqref{eq:tree-cstr} between the node $(d,i)$ and its child node $(d+1,j)$.
\changed{For brevity, the arguments of $L$ are omitted whenever they are clear from the context.}
The Lagrangian can be rearranged as a sum of multiple Lagrangians for each layer $d$ of the tree network as:
\begin{equation} \label{eq:tree-lag}
    L = \sum_{d=1}^{D}  \left( L_d^{\mathrm p} (x_d, x_{d+1}, y_d) + L_d^{\mathrm e} (x_d) \right),
\end{equation}
where $L_D^{\mathrm p} = 0$, $L_1^{\mathrm e} = 0$, $L_d^{\mathrm p}(x_d, x_{d+1}, y_d) = \sum_{i\in\calP_d}  \Big( f_{d,i}(x_{d,i}) \!+ \!
     y_{d,i}^\top (A_{d,i} x_{d,i} \!+\! B_{d,i} [x_{d+1,j}]_{j\in \calC_{d,i}} \!-\! c_{d,i}) 
     + 
    \frac{\rho}{2} \Vert  A_{d,i} x_{d,i} + B_{d,i} [x_{d+1,j}]_{j\in \calC_{d,i}} \!-\! c_{d,i} \Vert_2^2 \Big)$ for $1 \leq d \leq D\!-\!1$, 
    and $L^{\mathrm e}_{d}(x_d) = \sum_{i\in\calL_d}  f_{d,i}(x_{d,i})$ for $2 \leq d \leq D.$

\subsection{Hierarchical ADMM (hADMM) algorithm}
\label{sec:hADMM-algorithm}

The \emph{key idea} of our proposed hADMM algorithm is to hierarchically solve primal variables top-down from the root node to the leaf nodes, then update the dual variables bottom-up from the leaf nodes to the root node.
Specifically,  \changed{inspired by the Jacobi-type ADMM} \cite{linGlobalLinear2015}, we formulate the hADMM 
 for the hierarchical optimization problem \eqref{eq:prob-tree-opt} into a forward pass and a backward pass, \changed{along the tree structure}, as follows.

\noindent\textit{Forward:}
\begin{equation}
  \begin{split}
    x^{k+1}_1 &= \displaystyle\argmin_{x_1} L(x_1, x_2^k,\dots, x_D^k; \changed{y_1^k,\dots,y_{D-1}^k})
    \\
    x^{k+1}_2 &= \displaystyle\argmin_{x_2} L(x_1^{k+1},x_2, x_3^k, \dots, x_D^k; \changed{y_1^k,\dots,y_{D-1}^k})
    \\
    \cdots &
    \\
    x^{k+1}_{D} &= \displaystyle\argmin_{x_D} L(x^{k+1}_1, \dots, x_{D-1}^{k+1}, x_D; \changed{y_1^k,\dots,y_{D-1}^k})
  \end{split}
  \label{eq:forward}
\end{equation}

\noindent\textit{Backward:}
\begin{equation}
  \begin{split}
    y_{D-1}^{k+1} &= y_{D-1}^k \!+\!  \rho(A_{D-1} x_{D-1}^{k+1} \!+\! B_{D-1} x_D^{k+1} \!-\! c_{D-1}) \\
    y_{D-2}^{k+1} &= y_{D-2}^k \!+\!  \rho(A_{D-2} x_{D-2}^{k+1} \!+\! B_{D-2} x_{D-1}^{k+1} \!-\! c_{D-2}) \\
    \cdots & \\
    y_1^{k+1} &= y_1^k \!+\!  \rho(A_1 x_1^{k+1} + B_1 x_2^{k+1} - c_1)
  \end{split}
  \label{eq:backward}
\end{equation}
\changed{In the forward pass \eqref{eq:forward}, 
each subproblem minimizes the full augmented Lagrangian $L$ over one block $x_d$ while all the other blocks are fixed at their latest values.
Since only the terms of \eqref{eq:tree-lag} that contain $x_d$ affect that minimization, each update reduces to the layer-wise expressions \eqref{alg:root}--\eqref{alg:parent} derived below.}
It is noted  that the computation and communication between nodes in the forward and backward passes are distributed and follow the tree structure.
Indeed, let us consider any node $(d,i)$. 
The forward pass \eqref{eq:forward} results in three cases as follows. 

\noindent
(i) If $(d=1,i=1)$ is the root node, where $x_1 = x_{1,1}$:
\begin{align} 
  x^{k+1}_{1,1} 
  &= \argmin_{x_{1,1}} L_1^{\mathrm p} (x_1, x_2^k, y_1^k)
    \nonumber \\
  &= \argmin_{x_{1,1}}  f_{1,1}(x_{1,1}) \!+\!
    (y_{1,1}^{k})^\top \big( A_1 x_1 \!+\! B_1 x_2^k \!-\! c_1 \!\big)
    \nonumber \\
  &\qquad\qquad\qquad +  \frac{\rho}{2} \Vert A_1 x_1 \!+\! B_1 x_2^k  \!-\! c_1 \Vert_2^2.
\label{alg:root}
\end{align}

\noindent
(ii) If $(d,i)$ is a leaf node with the parent node $(d-1,p)$: % where $\calC_{d,i} = \emptyset$: 
\begin{align}
    x^{k+1}_{d,i} &= \argmin_{x_{d,i}} L_{d-1}^{\mathrm p}(x_{d-1}^{k+1}, x_d, y_{d-1}^k) + L_d^{\mathrm e}(x_d)
    \nonumber\\
    &=\argmin_{x_{d,i}}     f_{d,i}(x_{d,i})  \nonumber
    \\
    &\quad +\! (y_{d-1,p i}^k)^\top\!( A_{d-1,p i}x_{d-1,p}^{k+1} \!+\! B_{d-1,p i} x_{d,i} \!-\! c_{d-1,p i}) 
    \nonumber\\
    &\quad +\! 
     \frac{\rho}{2}\Vert  A_{d-1,p i}x_{d-1,p}^{k+1} \!+\! B_{d-1,p i} x_{d,i} \!-\! c_{d-1,p i} \Vert_2^2 \notag
     \\
    &= \changed{ \argmin_{x_{d,i}} f_{d,i}(x_{d,i}) \!+\! \frac{\rho}{2} \Vert B_{d-1,p i} x_{d,i} \!+\! q_{d-1,pi}^{k+1} \Vert_2^2}.
    \label{alg:leaf}
\end{align}
\changed{where $q_{d-1,pi}^{k+1} = A_{d-1,pi}x_{d-1,p}^{k+1} + \frac{1}{\rho} y_{d-1,pi}^{k} - c_{d-1,pi}$.}

\noindent
(iii) If $(d,i)$ is a non-leaf node with the parent node $(d-1,p)$: % where $d \neq 1$ and $\calC_{d,i}\not \neq \emptyset$:
\begin{align}
    \!\!x^{k+1}_{d,i} &= \argmin_{x_{d,i}} L_{d-1}^{\mathrm p} (x_{d-1}^{k+1}, x_d, y_{d-1}^k) + L_d^{\rm p} (x_d, x_{d+1}^k, y_d^k)
    \nonumber \\
    &= \argmin_{x_{d,i}}  f_{d,i}(x_{d,i}) + \changed{\zeta_{d,i}^k(x_{d,i})}
    \nonumber \\
    &\quad +\!  (y_{d-1,p i}^k)^\top\!(A_{d-1,p i}x_{d-1,p}^{k+1} \!+\! B_{d-1,p i} x_{d,i} \!-\! c_{d-1,p i})
    \nonumber \\
    &\quad +\! \frac{\rho}{2}\Vert  A_{d-1,p i}x_{d-1,p}^{k+1} \!+\! B_{d-1,p i} x_{d,i} \!-\! c_{d-1,p i} \Vert_2^2 \notag
    \\
    % &= \changed{\argmin f_{d,i}(x_{d,i})  + g_{d,i}^k(x_{d,i}) } \notag
    % \\
    % &\qquad \qquad \qquad \qquad + \! \changed{\frac{\rho}{2}\Vert  B_{d-1,p i} x_{d,i} + q_{d-1,p i}^{k+1}  \Vert_2^2}\text,
    &= \changed{\argmin f_{d,i}(x_{d,i}) +\frac{\rho}{2} \sum_{j\in\calC_{d,i}} 
    \Vert A_{d,ij} x_{d,i} + q_{d,ij}^k \Vert_2^2}  \notag
    \\
    &\qquad\qquad\qquad\qquad 
    \changed{+ \frac{\rho}{2}\Vert  B_{d-1,p i} x_{d,i} + q_{d-1,p i}^{k+1} \Vert_2^2}\text,
    \label{alg:parent}
\end{align}
\changed{
where $\zeta_{d,i}^k(x_{d,i}) = \sum_{j\in\calC_{d,i}}  \big( (y_{d,ij}^k)^\top ( A_{d,ij} x_{d,i} \!+\! B_{d,ij} x_{d+1,j}^k \!-\! c_{d,ij}) + \frac{\rho}{2}\Vert  A_{d,ij} x_{d,i} \!+\! B_{d,ij} x_{d+1,j}^k \!-\! c_{d,ij} \Vert_2^2 \big)$
and $q_{d,ij}^k = B_{d,ij} x_{d+1,j}^k \!+\! \frac{1}{\rho} y_{d,ij}^k \!-\! c_{d,ij}$
}

Based on these three cases, we summarize the hADMM algorithm in Algorithm~\ref{alg:hier}, which is a distributed implementation of the forward and backward passes \eqref{eq:forward} and \eqref{eq:backward}.

\changed{The \textbf{novelty} of the hADMM algorithm is a tree-aware rearrangement of the augmented Lagrangian that makes each node update depend only on its parent, children, and their edge-associated dual variables.
Thus, hADMM is not merely a multi-block ADMM applied to a tree-structured constraint matrix; its subproblems and communications remain local to the hierarchy.
Unlike nADMM, each hADMM iteration requires only one top-down primal pass and one bottom-up dual pass, without nested lower-level ADMM solves.}

\begin{algorithm}[tb!]
  \caption{{Distributed Implementation of the hADMM for solving the hierarchical optimization problem \eqref{alg:hier}}.}
  \label{alg:hier}
    \begin{algorithmic}[1]
      % \Require Tree network of agents.
      \Function{hADMM}{$(d,i), q$}
        \If{$(d,i)$ is leaf node}
          \State{Compute $x_{d,i}$ by \eqref{alg:leaf} \changed{with $q_{d-1,pi}^{k+1} = q$}}
          \State{Send $x_{d,i}$ to parent $(d-1,p)$}
          \State{\Return}
        \EndIf
        %\Else
          \If{$(d,i) = (1,1)$} \Comment{root node}
            \State{Compute $x_{1,1}$ by \eqref{alg:root}}
          \Else
            \State{Compute $x_{d,i}$ by \eqref{alg:parent} \changed{with $q_{d-1,pi}^{k+1} = q$}}
            \State{Send $x_{d,i}$ to parent $(d-1,p)$}
          \EndIf
          \For{$j \in \calC_{d,i}$} \Comment{in parallel}
            \State{$q_{d,ij} = A_{d,ij}x_{d,i} + \frac{1}{\rho} y_{d,ij} - c_{d,ij}$}
            \State{Send $q_{d,ij}$ to $(d+1,j)$} \Comment{to child $(d+1,j)$}
            \State $x_{d+1,j} \!=\! \Call{hADMM}{(d+1,j),q_{d,ij}}$ \Comment{from child}
            \State{$r_{d,ij} = A_{d,ij} x_{d,i} + B_{d,ij} x_{d+1,j} - c_{d,ij}$}
            \State{Update $y_{d,ij} \leftarrow y_{d,ij} + \rho r_{d,ij}$}
          \EndFor
      \EndFunction
      \Repeat
        \State 
          $\Call{hADMM}{(1,1), []}$ \Comment{Solve \eqref{eq:prob-tree-opt} at root node}
      \Until{convergence as in Section \ref{sec:stop-crit}}
    \end{algorithmic}
\end{algorithm}

\subsection{Stopping criterion}
\label{sec:stop-crit}
\changed{
A non-leaf node $(d,i)$ can compute  $\bar{r}_{d,i}^{k+1} = \max_{j\in \calC_{d,i}} \Vert r_{d,ij}^{k+1} \Vert_\infty$ and  $\bar{s}_{d,i}^{k+1} = \max_{j\in \calC_{d,i}} \Vert s_{d,ij}^{k+1} \Vert_\infty$ where $r_{d,ij}^{k+1} = A_{d,ij} x_{d,i}^{k+1} + B_{d,ij} x_{d+1,j}^{k+1} - c_{d,ij}^{k+1}$ is the primal residual and $s_{d,ij}^{k+1} = \rho A_{d,ij}^\top B_{d,ij} (x_{d+1,j}^{k+1} - x_{d+1,j}^{k})$ is the dual residual \cite[Chapter 3.3]{boyd2011distributed} for its children $j\in \calC_{d,i}$.
Inspired by the stopping criterion used in \cite{stellato2020osqp} with $\epsilon_\text{rel}=0$,
the stopping criterion of the hADMM is given by
\begin{align}
  R^{k+1} = \max_{d \in \bbN_{D-1}} \max_{i\in \calP_d}  R_{d,i}^{k+1} \leq \varepsilon_\mathrm{abs} \text,
  \label{eq:stop-criterion}
\end{align}
where $\varepsilon_\mathrm{abs} > 0$ is a predefined tolerance and $R_{d,i}^{k+1} = \max( \bar{r}_{d,i}^{k+1}, \bar{s}_{d,i}^{k+1})$.
Note that this stopping criterion can be checked in a distributed manner as follows.
In the backward pass, each non-leaf node $(d,i)$ computes its maximum residual $\bar R_{d,i}^{k+1} = \max\{R_{d,i}, \max_{j\in \calC_{d,j}}\{ \bar R_{d+1,j}^{k+1}\}\}$, where $\bar R_{d+1,j}^{k+1} = 0$ if $(d+1,j)$ is a leaf node, then sends $\bar R_{d,i}^{k+1}$ to its parent node.
Recursively until $d = 1$, the root has $ \bar R_{1,1}^{k+1} = \max\{R_{1,1}, \max_{j\in \calC_{1,1}}\{ \bar R_{2,j}^{k+1}\}\} = R^{k+1}$ as in \eqref{eq:stop-criterion}.}

\changed{%
\begin{remark}
  In hADMM, each parent node $(d,i)$ stores the constraint parameters $A_{d,ij}$, $B_{d,ij}$, and $c_{d,ij}$ and the dual variable $y_{d,ij}$ for all $j\in\calC_{d,i}$. It sends each child a query $q_{d,ij}$ in the forward pass and receives the child variable $x_{d+1,j}$ in the backward pass.
\end{remark}
\begin{remark}
  Upon receiving $q_{d,ij}$, child node $(d+1,j)$ immediately updates $x_{d+1,j}$ using \eqref{alg:leaf} or \eqref{alg:parent} and returns the updated value.
  This trigger-response mechanism allows nodes at the same level to solve their subproblems without waiting for one another.
\end{remark}%
}

\section{Convergence Analysis of hADMM}
\label{sec:convergence}

This section provides theoretical guarantees for the convergence of the proposed hADMM in  Algorithm \ref{alg:hier} and  a theoretical analysis of its convergence rate.
We consider the following scenarios in our convergence analysis, 
based on \cite{cao2019dynamic, deng2017parallel, deng2016global}, that are commonly assumed in the literature.
The convergence analysis results are summarized in Table~\ref{tab:con_res}.

\begin{table}[!t]
  \centering
  \caption{Scenarios and convergence results of the hADMM.}
  \label{tab:con_res}
  \setlength{\tabcolsep}{3pt}
  \begin{tabular}{@{}p{0.50\columnwidth}|p{0.46\columnwidth}@{}}
    \textbf{Scenario} & \textbf{Theoretical results}
    \\
    \hline
    \changed{\ref{scn:convex_rank}: $f_1,f_2,\dots,f_D$ are convex,  $A_1, B_1, \dots,B_D$ are full column rank} & \changed{Theorem \ref{thm:converge}: residual convergence, bounded primal and dual sequences.}
    \\
    \hline
    \changed{\ref{scn:str_convex}: $f_1,f_2,\dots,f_D$ are differentiable strongly convex} & Theorem \ref{thm:convergence_strcon}: convergence
    \\
    \hline
    \changed{\ref{ass:Lipschitz}: $f_1,f_2,\dots,f_D$ are differentiable strongly convex, $f_2,\dots,f_D$ have Lipschitz-continuous gradients} & Theorem \ref{thm:convergence_rate}: linear convergence rate.
  \end{tabular}
\end{table}

\begin{scenario} \label{scn:convex_rank}
    Functions $f_1, f_2,\dots$ and $f_D$ are proper,  l.s.c and convex, 
    \changed{matrix $A_1$ has full column rank, and matrices $B_1, \dots, B_{D-1}$ have full column rank.}
\end{scenario}

\begin{scenario} \label{scn:str_convex}
  \changed{
  Functions $f_1, f_2,\dots, f_D$ are differentiable strongly convex with moduli $\alpha_d > 0$,} that is, for all $z,z^\prime \in {\rm Dom}(f_d)$,
  \begin{math} 
    f_d(z) - f_d(z^\prime) \geq \nabla f_d^\top(z^\prime) (z-z^\prime) + \alpha_d \Vert z - z^\prime \Vert_2^2\text.
  \end{math}
\end{scenario}

\subsection{Convergence of hADMM}
\label{sec:convergence_hADMM}

\changed{
We assume that the convex optimization problem \eqref{prob:chain_mul_block_opt} has at least one strictly feasible point (Slater's Condition) and admits an optimal solution.
Under these conditions, solving \eqref{prob:chain_mul_block_opt} is equivalent to  finding $([x_d^\star]_{d \in \bbN_D}, [y_d^\star]_{d \in \bbN_{D-1}})$ 
satisfying the generalized Karush-Kuhn-Tucker (KKT) condition for 
\eqref{eq:prob-tree-opt} as follows:
}
\begin{subequations}
  \label{eq:prf-KKTs}
  \begin{align}
    &0 \in \partial f_1 (x_1^*) + A_1^\top y_1^*
      \label{thm:prf_KKT1}
    \\
    &0 \in \partial f_D (x_D^*) + B_{D-1}^\top y_{D-1}^*
      \label{thm:prf_KKT2}
    \\
    &0 \in \partial f_d (x_d^*) + A_d^\top y_d^* + B_{d-1}^\top y_{d-1}^*, \; d=2,\dots,D\!-\!1
      \label{thm:prf_KKT3}
    \\
    &A_d x_d^* + B_d x_{d+1}^\star = c_d,~d =1,2\dots, D-1\text.
      \label{thm:prf_KKT4}
  \end{align}
\end{subequations}

Let $\calO^\star$ be the set of all solutions of the KKT condition \eqref{thm:prf_KKT1}-\eqref{thm:prf_KKT4}.
\changed{The following theorem presents convergence of the proposed hADMM in Algorithm \ref{alg:hier} under Scenario~\ref{scn:convex_rank}.}

\begin{theorem} \label{thm:converge}
  \changed{Under Scenario \ref{scn:convex_rank}, the sequence $\{[x_d^k]_{d\in \bbN_D}, [y_d^k]_{d\in \bbN_{D-1}}\}_k$ obtained by Algorithm \ref{alg:hier} is bounded, the primal residuals satisfy $r_d^k \to 0$ for all $d\in\bbN_{D-1}$, and every cluster point of the sequence belongs to $\calO^\star$.}
  \changed{If, in addition, $\calO^\star$ is a singleton, then the whole sequence converges to the unique point $([x_d^\star]_{d \in \bbN_D}, [y_d^\star]_{d \in \bbN_{D-1}})$ in $\calO^\star$.}
\end{theorem}
\begin{proof}
    See Appendix~\ref{apdx:thm1_prf}.
\end{proof}

\changed{Similarly, Theorem~\ref{thm:convergence_strcon} establishes convergence properties of Algorithm~\ref{alg:hier} under Scenario~\ref{scn:str_convex}. 
}

\begin{theorem} \label{thm:convergence_strcon}
  \changed{Under Scenario \ref{scn:str_convex}, the primal sequence $\{[x_d^k]_{d\in \bbN_D}\}_k$ obtained by Algorithm \ref{alg:hier} converges to the unique minimizer $[x_d^\star]_{d \in \bbN_D}$.
  }
\end{theorem}
\begin{proof}
\ifdefined\SubmittedPaper
    The proof of Theorem \ref{thm:convergence_strcon} can be derived by the proof of Theorem \ref{thm:converge} and properties of strongly convex functions. See the full proof in the full-length paper \cite[Appendix~C]{hADMMfullpaper}.
\else
    See Appendix~\ref{apdx:thm2_prf}.
\fi
\end{proof}

\subsection{Convergence rate of hADMM}

\changed{To analyze the convergence rate of hADMM, we consider the following scenario, which is similar to but stricter than Scenario \ref{scn:str_convex} by assuming Lipschitz continuity of most gradients.}
\changed{
\begin{scenario} \label{ass:Lipschitz}
    $\Vert B_d\Vert_2 \neq0$ for all $d\in\bbN_{D-1}$. Functions $f_1, f_2,\dots f_D$ are differentiable strongly convex with moduli $\alpha_d > 0$, and $\nabla f_2,\dots \nabla f_D$ are Lipschitz continuous with constant $l_d$, that is, for all $z,z^\prime \in {\mathrm Dom}(f_d)$,
    \begin{align}
        \Vert \nabla f_d(z) - \nabla f_d(z^\prime) \Vert_2 \leq l_d \Vert z - z^\prime \Vert_2.
    \end{align}
\end{scenario}
}

\changed{Theorem~\ref{thm:convergence_rate} below shows that under Scenario~\ref{ass:Lipschitz}, hADMM has a linear convergence rate.}

\begin{theorem}[\changed{Linear convergence rate}] \label{thm:convergence_rate} 
  Assume Scenario \ref{ass:Lipschitz}. Choose $t\in (0,1)$ and $\changed{0< \sigma < \bar \sigma = \min_{d\in \bbN_{D-1}}} \left\{\frac{2t\alpha_{d+1}}{\Vert B_d \Vert_2^2}\right\}$.
  \changed{Assume the following semidefinite program (SDP) is feasible and let $\gamma^\star>0$ be its optimal value:}
  \begin{subequations}
    \label{thm:sdp}
    \begin{align}
      \gamma^\star = \max_{\gamma > 0} \;&\gamma \\
      \mathrm{s.t.}\; &\Xi - \gamma \Theta \succeq 0\text, \label{thm:rate_con}
    \end{align}
  \end{subequations}
\changed{
where $\Xi$ is a symmetric matrix defined in Appendix~\ref{apdx:thm3_prf} and \(\Theta = {\diag}\left(\frac{1}{\rho} I_s, (\rho + \sigma) I_s, 0_s, 0_s \right)\).
Define the Lyapunov-type quantity
\(\hat{V}^k = \sum_{d=1}^{D-1} \frac{1}{\rho} \Vert y_d^k - y_d^\star \Vert_2^2  + (\rho + \sigma) \Vert B_d( x_{d+1}^k - x_{d+1}^\star) \Vert_2^2\).
Then, $\hat{V}^k$ satisfies}
\begin{align}
    (1+\gamma^\star)\hat V^{k+1} \leq \hat V^k \text.
    \label{eq:lya_ctrc}
\end{align}
\changed{Moreover, the sequence $\{[x_d^k]_{d\in \bbN_D}\} $ converges linearly to the unique minimizer $[x_d^\star]_{d\in \bbN_D}$, that is, there exists a positive constant $\calC$ such that:
\begin{align} \label{eq:lin-con}
  \sum_{d=1}^{D} \Vert x_d^{k+1} - x_d^\star \Vert_2 \leq \calC \Big( \frac{1}{\sqrt{1+\gamma^\star}}\Big)^k \text.
\end{align}
}
\end{theorem}
\begin{proof}
    % \ifdefined\SubmittedPaper
    See Appendix~\ref{apdx:thm3_prf}.
\end{proof}

The feasibility of \eqref{thm:sdp} is crucial in Theorem \ref{thm:convergence_rate}, \changed{for which a sufficient condition is provided by the following proposition.}

\begin{proposition} \label{prop:fullrank}
    \changed{
    Under Scenario~\ref{ass:Lipschitz} and the parameter choices of
    Theorem~\ref{thm:convergence_rate}, suppose the block lower-bidiagonal matrix
    \begin{align} \label{eq:Lambda}
      \Lambda =
      \begin{bmatrix}
        B_1 \\
        A_2 & B_2 \\
            & \ddots & \ddots \\
            &        & A_{D-1} & B_{D-1}
      \end{bmatrix}
    \end{align}
    has \changed{full row rank}. Then, the SDP~\eqref{thm:sdp} is feasible.
    }
\end{proposition}
\begin{proof}
    \ifdefined\SubmittedPaper
    See \cite[Appendix~E]{hADMMfullpaper}.
    \else
    See Appendix~\ref{prop:proof_rate}.
    \fi
\end{proof}

\begin{remark} \label{rm:bilevel}
\changed{The SDP~\eqref{thm:sdp} in Theorem \ref{thm:convergence_rate} maximizes the lower bound of the linear convergence rate parameter $\gamma^\star$ with respect to the network structure, constraints, and penalty parameter $\rho$.
It provides a criterion for designing a tree network that maximizes the contraction rate of $\hat V^k$.}
Furthermore, the choice of $(t,\sigma)$ will affect $\gamma^\star(t, \sigma)$.
To find the largest $\gamma^\star$, we can consider the bilevel optimization $\max{\gamma^\star(t, \sigma)}\;\mathrm{s.t.}\; t\in (0,1), \sigma \in (0, \bar \sigma)$.
\end{remark}

\begin{remark}
If $D=2$, the network $\calG$ has a star topology and the condition in Proposition~\ref{prop:fullrank} reduces to the requirement that $B_1$ has full row rank, % is a full-column-rank matrix,
which aligns with the conditions in \cite[Assumption 4]{cao2019dynamic} and \cite[Scenario 2]{deng2016global}.
\end{remark}

\section{Numerical Examples}
\label{sec:examples}

This section presents three numerical examples to validate the proposed hADMM method in 
Algorithm \ref{alg:hier} and convergence results in Section \ref{sec:convergence}. 
To compare our approach with existing methods, we consider two baselines:
\begin{itemize}
    \item Baseline 1 is the nADMM described in  Section \ref{sec:nADMM} and   \ifdefined\SubmittedPaper
        \cite[Alogrithm~2]{hADMMfullpaper},
        \else
        Algorithm \ref{alg:nADMM},
        \fi which has been employed in \cite{khakiHierarchicalADMMBased2018, zhangCooperativeLoadScheduling2020a, braunHierarchicalDistributedADMM2018}.
    \item Baseline 2 is the fADMM outlined in Section \ref{sec:flattenADMM}, which flattens the tree network for direct communication (via relays) between the root node and all the other nodes to apply the standard ADMM.
\end{itemize}
This paper considers the following three metrics to evaluate and compare the performance of the 
methods:
\begin{itemize}
\item \textit{Maximum number of iterations in a node ($\mathtt{MaxNodeIters}$):}
  This metric evaluates the convergence speed and the computational load of each method.  
  In both hADMM and fADMM, all nodes perform the same number of iterations. 
  In contrast, nADMM exhibits varying iteration counts across nodes 
  due to the recursive structure described in 
  \ifdefined\SubmittedPaper
  \cite[Alogrithm~2]{hADMMfullpaper}.
  \else
  Algorithm \ref{alg:nADMM}.
  \fi
\item \textit{Maximum number of scalar values sent by a node ($\mathtt{MaxNodeVals}$):}
  This metric quantifies the maximum total communication load experienced by a node, which may become 
  a \changed{\textbf{communication bottleneck}} due to excessive data transmission during the optimization process.
\item \textit{Total number of scalar values sent in the network ($\mathtt{TotalVals}$):}
  This metric measures the overall communication cost across the network when solving the optimization problem \eqref{eq:prob-tree-opt}. 
  It reflects the \changed{\textbf{average communication load}} and is approximately proportional to the 
  total number of bits transmitted.
\end{itemize}

\changed{
The communication cost metrics \texttt{MaxNodeVals} and \texttt{TotalVals} count the scalar values sent by each node, whether originating from or relayed by that node.
In hADMM, these values comprise the queries and responses exchanged between nodes.
Because the corresponding vectors may have different dimensions, we count their individual elements.
The root exchanges messages only with its children; intermediate nodes receive queries from their parent and responses from their children, and send responses to their parent and queries to their children; leaf nodes only receive queries and return responses.
In nADMM, the communication cost includes queries and responses in both the inner and outer loops.
In fADMM, messages from nodes not directly connected to the root are relayed along their paths to the root, with each relay counted as communication by the relaying node. Residual checks are excluded in all cases.
In addition to these metrics, we evaluate convergence and accuracy using the optimality gap
\(
g_\mathrm{opt}(k)
= \frac{\left|J^k-J^\star\right|}
{\max\!\left(\left|J^0-J^\star\right|,\delta\right)}
\),
where $\delta>0$ is a small constant.
}

\subsection{Disjoint problems over randomized tree networks
\label{sec:examples:disjoint}}

This experiment aims to thoroughly validate Algorithm \ref{alg:hier} under various randomly generated tree-network topologies. 
Specifically, we used the proposed hADMM (Algorithm \ref{alg:hier}), nADMM (\ifdefined\SubmittedPaper \cite[Alogrithm~2]{hADMMfullpaper} \else Algorithm \ref{alg:nADMM}\fi), and fADMM for solving the disjoint problem~\eqref{prob:disjoint} with three network configurations as $(D, N) = (3,10), (3,20)$ and $(5,20)$.
In each configuration, the algorithms are tested over 1000 randomized tree topologies to demonstrate the convergence behavior predicted by Theorem \ref{thm:converge} and to compare the performance of these methods.
In all cases, the matrices $B_d$ have full column rank. 
Each local cost function is convex (but not strongly convex), and is defined as
\begin{equation*}
\changed{f_{d,i}(x_{d,i}) = \Vert x_{d,i} - u_{d,i} \Vert_2^4 + w \Vert x_{d,i} \Vert_1,}
\end{equation*}
\changed{where $w = 2$ and $u_{d,i}$ is a real vector preassigned for each node.}
We can see that in this setup, Scenario \ref{scn:convex_rank} is satisfied.
Each leaf node has a scalar local variable.
\changed{%
Table~\ref{tab:rho-sen} reports penalty-parameter sensitivity for $(D,N)=(3,10)$.
Guided by this study, $\rho$ for each method is manually tuned for each network topology (\eg $(D,N)=(3,10)$, $(3,20)$, and $(5,20)$) to achieve the fastest convergence for the respective methods with  $\varepsilon_\mathrm{abs}=10^{-4}$.
}

\begin{table}
  \centering
  \changed{
  \caption{$\mathtt{MaxNodeIters}$ with $\varepsilon_\mathrm{abs} = 10^{-4}$ and different $\rho$ ($D=3, N=10$).}
  \label{tab:rho-sen}
  \begin{tabular}{|c|c|c|c|c|c|c|c|c|}
    % \hline
    $\rho$ &10 &40 &80 &100 &160 &200 &400 &1000
    \\
    \hline
    hADMM &164 &54 &30 &30  &38  &46  &85  &196
    \\
    \hline
    fADMM &97  &33 &24 &27  &37  &44  &84  &195
    \\
    \hline
    $\begin{matrix}
        \text{nADMM} \\ \text{(at the root)}
    \end{matrix}$ 
    &98  &34 &21 &22  &25  &44  &50  &114
    \\
    \hline
  \end{tabular}
  }
  % \vspace{-1em}
\end{table}
The convergence to an optimal solution of hADMM over 1000 tree networks with $(D=3, N=10)$ is shown in Fig.~\ref{fig:Conver_valid}.
\changed{
From Fig. \ref{scn:cost_func_error}, the optimality gap $g_\mathrm{opt}$ converges to zero as the iteration increases. 
}
These results validate the convergence proof established in Theorem \ref{thm:converge} for non-smooth cost functions.

Table~\ref{tab:disjoint} summarizes the results of the performance metrics of the algorithms. 
As expected, nADMM consistently has the highest $\mathtt{MaxNodeIters}$ across all network configurations.
This indicates that nADMM converges more slowly compared to fADMM and hADMM.
Besides, nADMM also has the largest $\mathtt{MaxNodeVals}$ and $\mathtt{TotalVals}$, which increase geometrically with the network depth, showing the highest and substantial communication costs among the methods.

In Table~\ref{tab:disjoint}, we also observe that fADMM exhibits the smallest $\mathtt{MaxNodeIters}$, as it is based on the star topology.  
However, its communication costs, measured by $\mathtt{MaxNodeVals}$ and $\mathtt{TotalVals}$, are higher than those of hADMM. 
This is because nodes in lower layers (\eg $d > 3$) incur more communication overhead when exchanging information with the root node, since all nodes along the paths between them must relay their messages.
Statistical comparisons of communication costs between fADMM and hADMM are provided in the boxplots in Fig.~\ref{fig:total_com}.
They clearly show that hADMM significantly reduces $\mathtt{TotalVals}$, with the differences becoming increasingly pronounced as the network depth increases from $D = 3$ to $D = 5$.

\changed{%
A common tolerance is insufficient to establish equivalent termination accuracy across the three methods, particularly with nADMM's recursive inner solves. 
With primal-dual stopping criterion in section~\ref{sec:stop-crit},
to supplement the comparisons at termination, we average the normalized optimality gap over 1,000 randomly generated tree topologies.
Fig.~\ref{fig:gap-com} plots this gap against $\mathtt{TotalVals}$ for hADMM and fADMM.
At comparable objective-gap levels, hADMM transmits fewer values, with a larger advantage as the network depth grows from $D=3$ to $D=5$.
}

\begin{table*}
  \centering
  \changed{
  \caption{Performance comparison for solving the disjoint problem \eqref{prob:disjoint} between the baseline methods (nADMM and fADMM) and the proposed hADMM over 1000 randomized tree networks with different depths ($D$) and numbers of nodes ($N$).
  }
  \label{tab:disjoint} 
  \begin{tabular}{|c|c|c|c|c|}
    \hline
    Performance metric & Network & nADMM &fADMM &hADMM \\
    ($\varepsilon_\mathrm{abs} = 10^{-4}$)&(depth and \# nodes) & (min) mean (max) & (min) mean (max) & (min) mean (max)\\
    \hline
    \multirow{3}{*}{$\left. \mathtt{MaxNodeIters} \right.$} 
    &D=3, N = 10 &(181) 193 (236)   &\textbf{(23) 24 (28)} &(23) 30 (33)\\
    &D=3, N = 20 &(191) 266 (404)  &\textbf{(27) 30 (33)} &(26) 33 (43)\\
    &D=5, N = 20 &(11729) 14071 (17146)  &\textbf{(30) 38 (52)} &(35) 49 (82)\\
    \hline
    \multirow{3}{*}{$\left. \mathtt{MaxNodeVals} \right.$} 
    &D=3, N = 10 &(880) 1273 (1952)   &(230) 360 (624)   &\textbf{(140) 300 (480)}\\
    &D=3, N = 20 &(1882) 3450 (6462)  &(621) 900 (1782)  & \textbf{(351) 595 (1224)}\\
    &D=5, N = 20 &(14692) 49344 (267600)  &(726) 1466 (4480)  &\textbf{(250) 575 (1848)}\\
    \hline
    \multirow{3}{*}{ $\left. \mathtt{TotalVals} \right.$} 
    &D=3, N = 10 &(2745) 3966 (5693)  &(690) 1008 (1344)  &\textbf{(560) 840 (960)} \\
    &D=3, N = 20 &(5690) 10664 (18378)  &(1728) 2580 (3564) &\textbf{(1288) 1922 (2574)}\\
    &D=5, N = 20 &(109252) 299272 (773312)  &(3168) 5688 (12800)  &\textbf{(1848) 3079 (7260)}\\
    \hline
  \end{tabular}
  }
  \vspace{-1.5em}
\end{table*}

\begin{figure}[!tb]
\vspace{-4pt}
\changed{
\centering
\subfloat[ Optimality gap \label{scn:cost_func_error}]{\includegraphics[width = 0.5\linewidth]{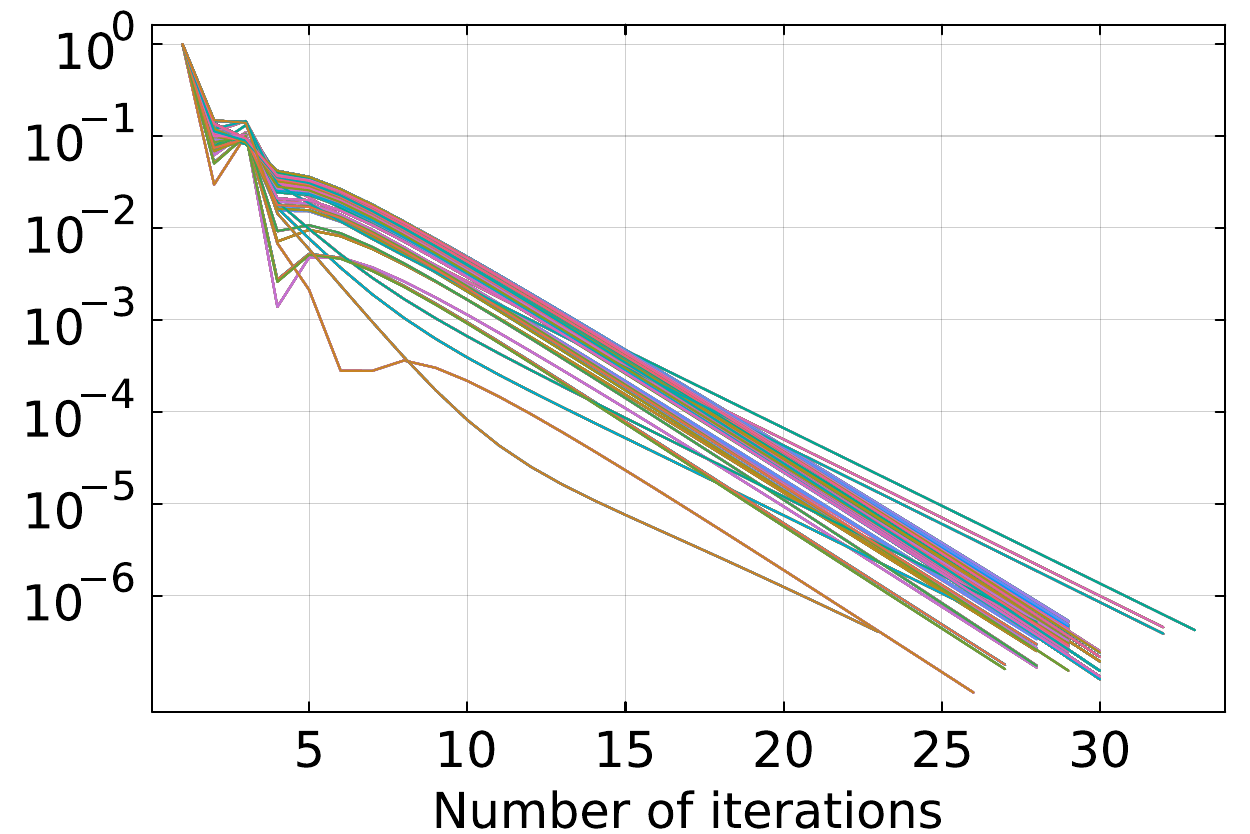}}
\subfloat[ Primal-dual residual \label{scn:residual}]{\includegraphics[width = 0.5\linewidth]{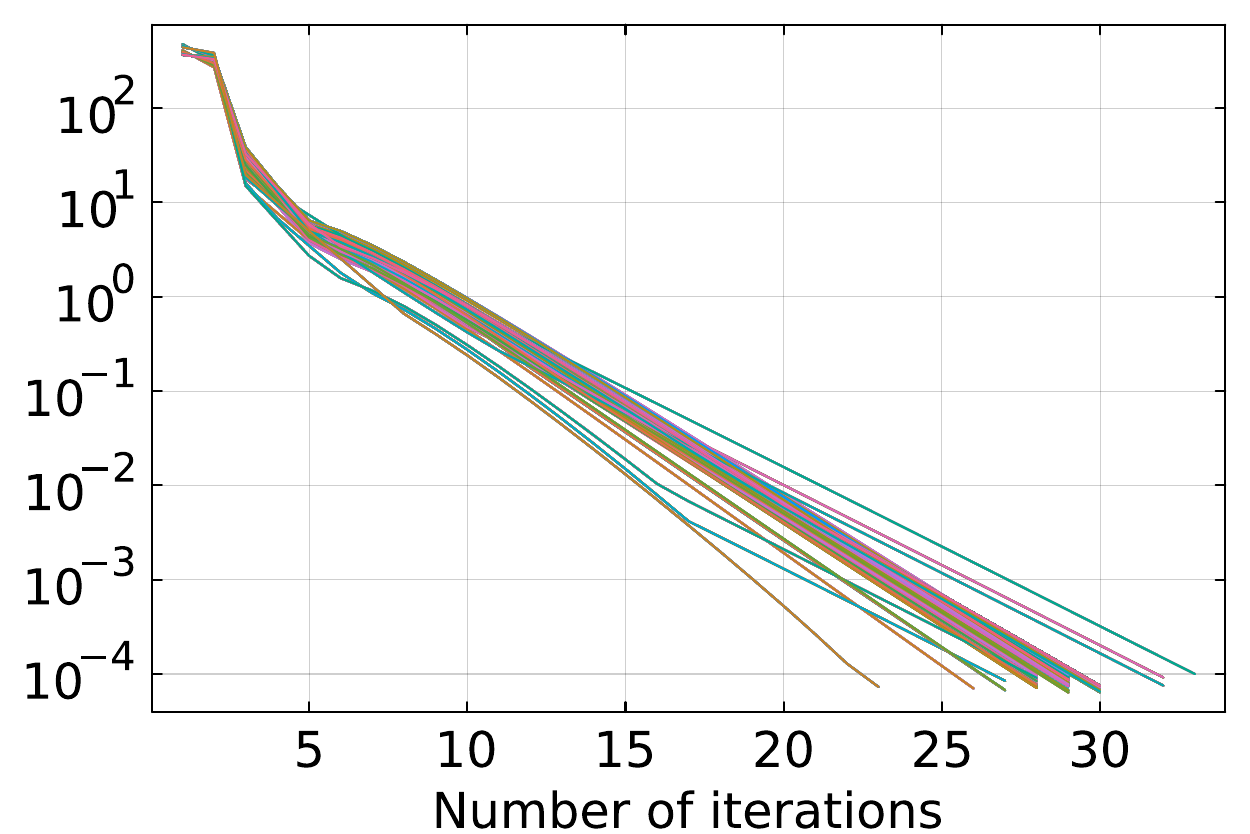}}
\caption{Convergence of hADMM to an optimal solution under 1000 randomized tree networks with ($D=3, N=10$).
}
\label{fig:Conver_valid}
}
\end{figure}

\begin{figure}[!tb]
\vspace{-4pt}
\changed{
\centering
\subfloat[$D=3, N=10$]{\includegraphics[width=0.33\linewidth]{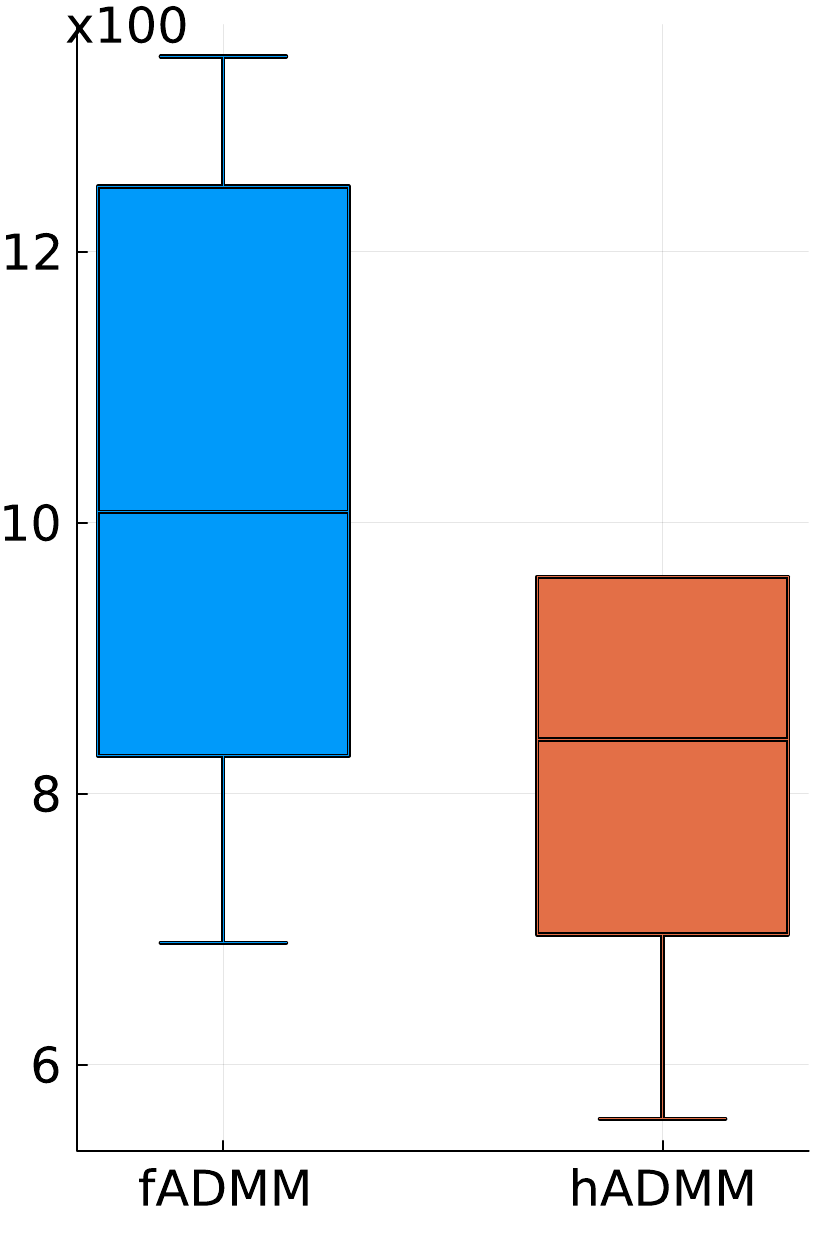}}
\subfloat[$D=3, N=20$]{\includegraphics[width=0.33\linewidth]{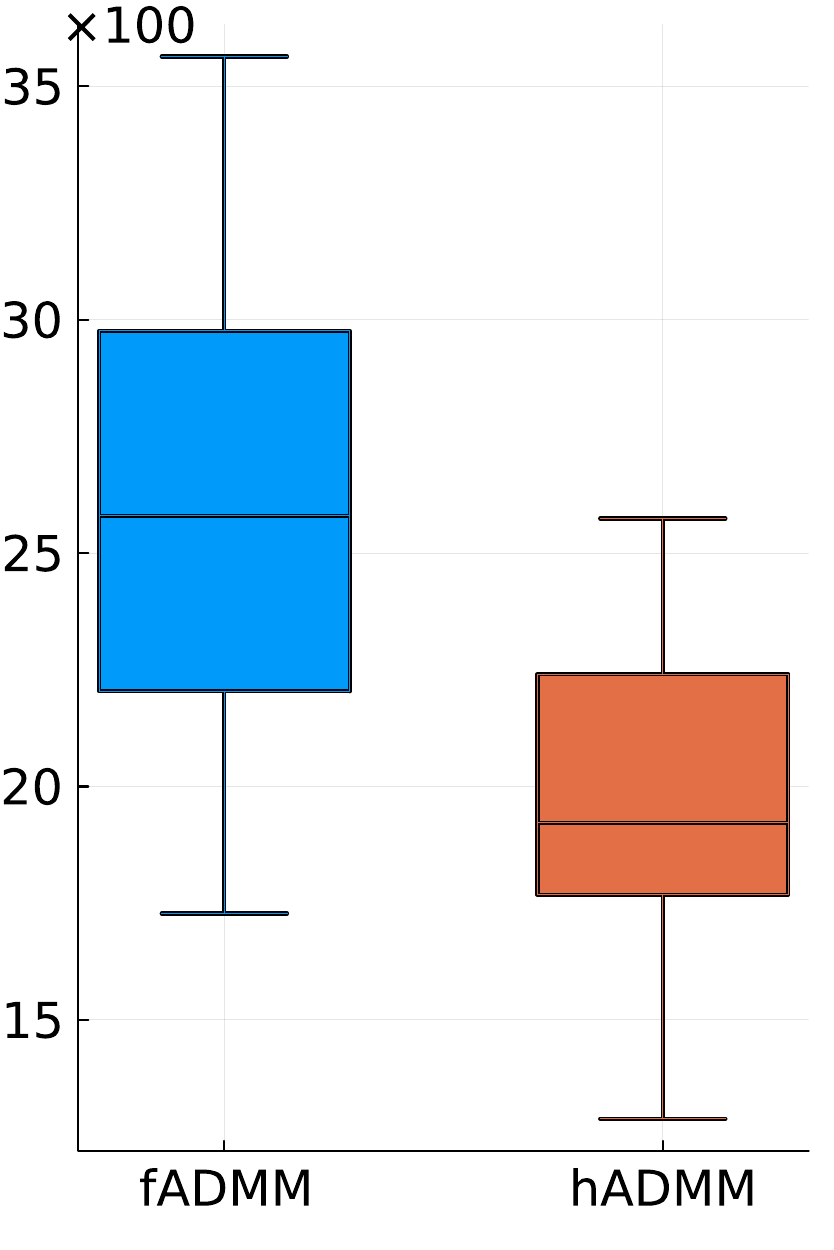}}
\subfloat[$D=5, N=20$]{\includegraphics[width=0.33\linewidth]{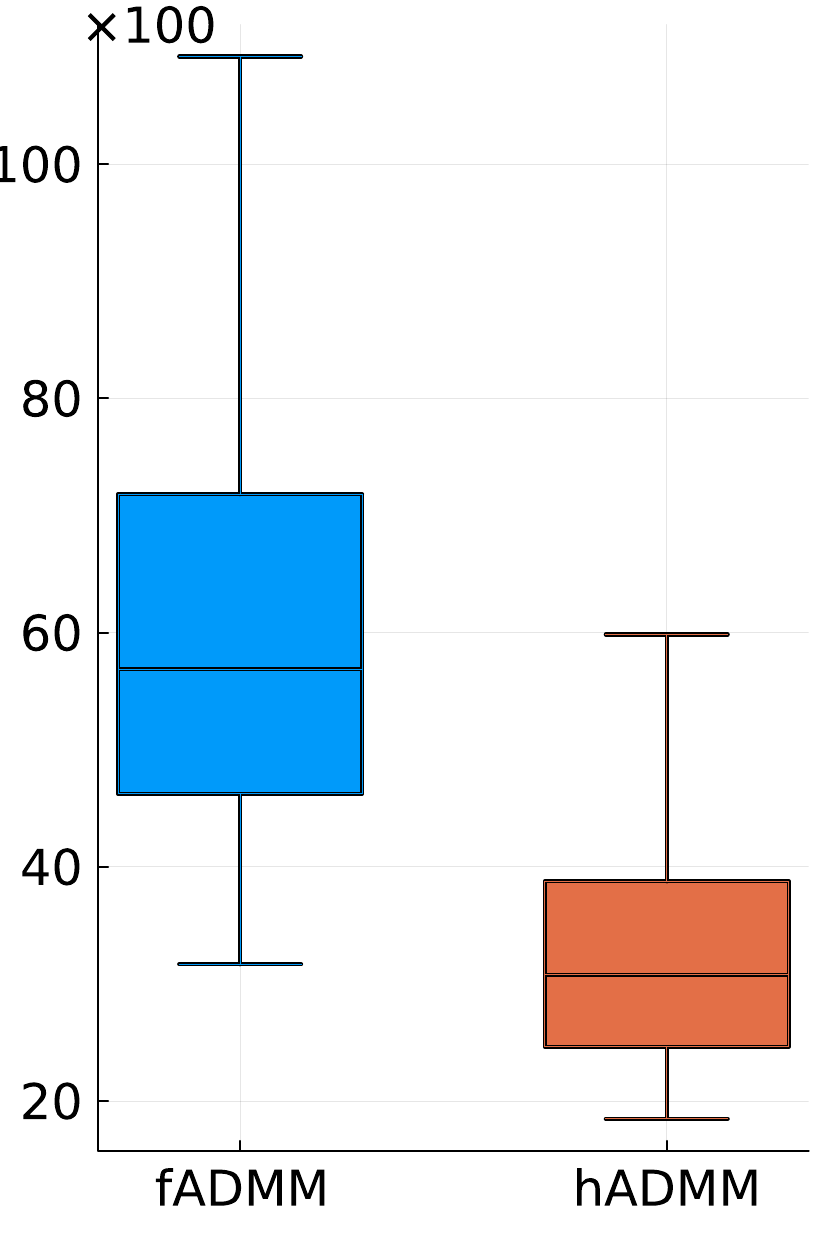}}
\caption{Boxplots of $\mathtt{TotalVals}$ of fADMM and hADMM under three network configurations, each is over 1000 randomized network topologies. \label{fig:total_com}}
}
\end{figure}

\begin{figure}[!tb]
\vspace{-4pt}
\centering
\changed{
\subfloat[$D=3, N=10$]{\includegraphics[width=0.5\linewidth]{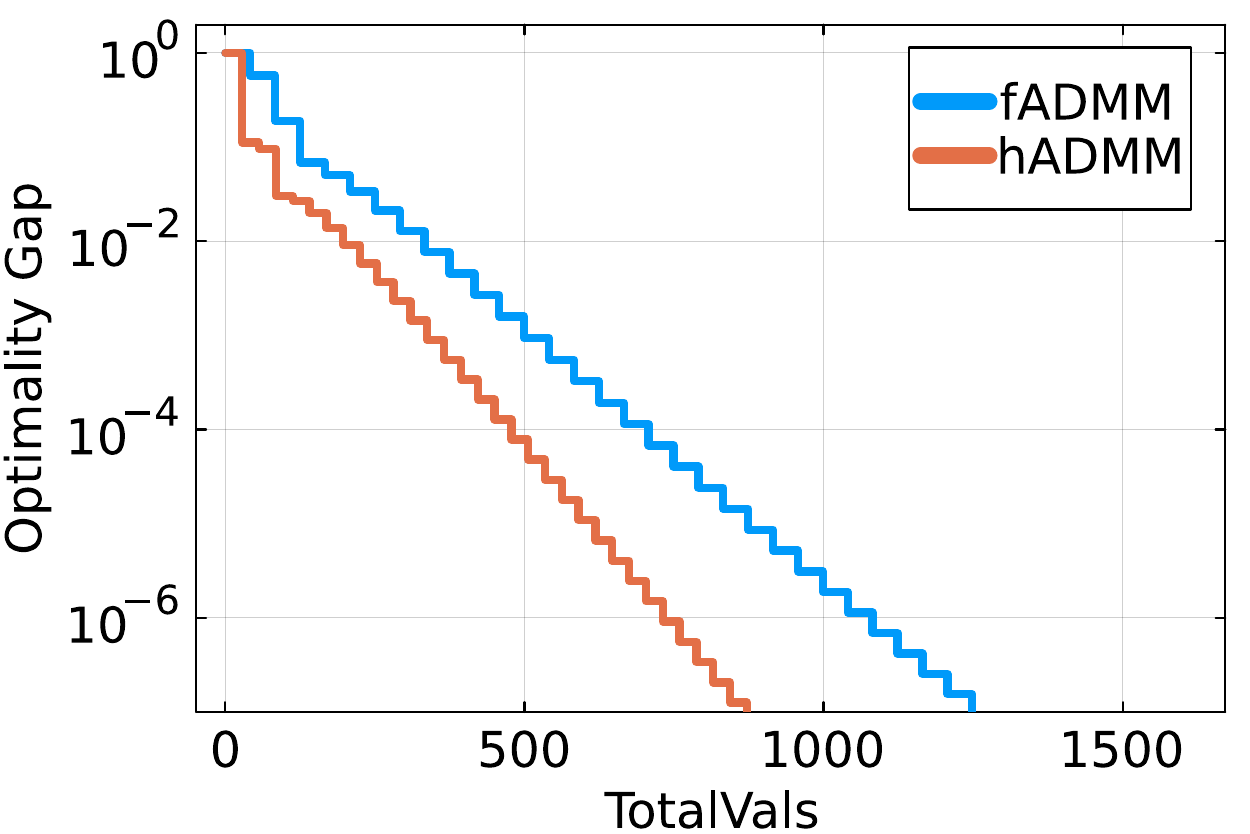}}  
\subfloat[$D=5, N=20$]{\includegraphics[width=0.5\linewidth]{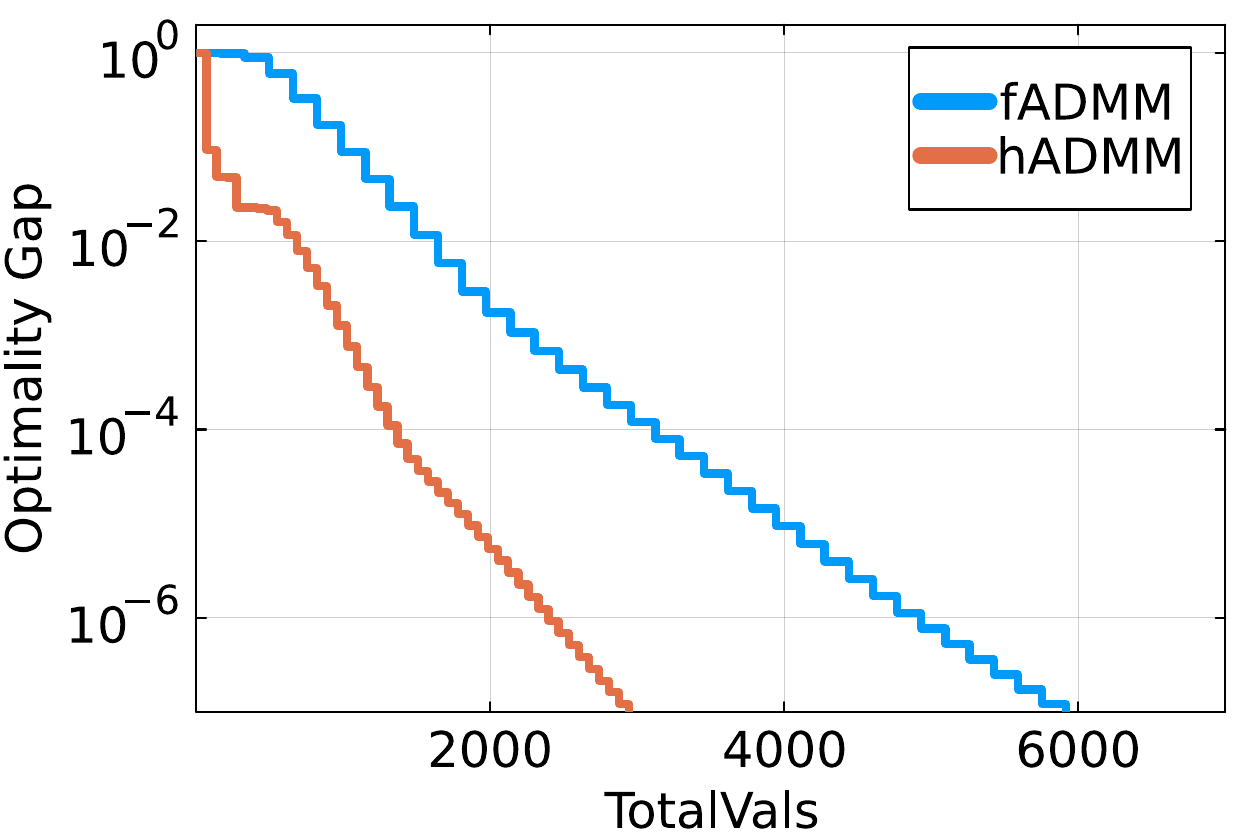}}  
\caption{Average optimality gap of 1000 randomized disjoint problems with respect to $\mathtt{TotalVals}$. \label{fig:gap-com}}
}
\end{figure}

\subsection{Power sharing problem}
\label{sec:examples:power-sharing}

Let us consider an electrical power distribution grid with a hierarchical structure described in the example in Section~\ref{sec:introduction}.
A power sharing problem is formulated for the grid, as presented in \cite{zhangCooperativeLoadScheduling2020a} and briefly summarized below. 
The utility company coordinates with $N_a$ aggregators, each managing multiple-device clusters to control power flows in the grid. 
Aggregators retain autonomy in managing and configuring these clusters. 
In real-world deployments, the power-flow control architectures and algorithms are 
secured between the aggregators and cluster layers and are not disclosed to the utility (root).
Therefore, the utility can only communicate with the aggregators, not directly with the clusters.
A power sharing problem optimally schedules the power allocations to all nodes in the network.
Solving the sharing problem, formulated as a hierarchical optimization problem \eqref{eq:prob-tree-opt}, enables the utility to coordinate power consumption across all aggregators to track a desired load profile while minimizing costs and satisfying all constraints at every level of the hierarchy.

For a time-dependent schedule with $h_t$ horizon steps, denote
$z_{ij} = [z_{ij}[t]]_{t = 1 \dots h_t} \in \RR_{>0}^{h_t}$
as the vector of powers allocated to cluster $j$ of aggregator $i$ over $h_t$ steps ahead. % $(t=1,2,\dots,h_t)$.
The \changed{aggregator} $i$ manages $N_{c,i}$ clusters.
We also denote $z_i[t] = \sum_{j=1}^{N_{c,i}} z_{ij}[t]$ as the power consumed by aggregator $i$ at time step $t$, and $z_i = [z_i[t]]_{t = 1 \dots h_t}$. 
The nonnegative element-wise vector $\beta \in \RR_{>0}^{h_t} $ stands for the amount of power obtained from the day-ahead wholesale market in order to satisfy the aggregator demand.
For the total power to track the load signal $\beta$, costs are incurred at all three levels and our objective is to minimize the following total cost:
\begin{align}
    \min_{z_i \succeq 0} \sum_{i=1}^{N_a} f_i(z_i) 
    + \lambda \Big\Vert \sum_{i=1}^{N_a} z_i - \beta \Big\Vert_2^2,
    \label{prob:sharing_exam}
\end{align}
where $f_i(z_i)$ is the cost function of aggregator $i$ and its clusters,
and $\lambda \Vert \sum_{i=1}^{N_a} z_i - \beta \Vert_2^2$ is the utility-level penalty for deviation from the aggregate load signal with penalty factor $\lambda > 0$. 
Each aggregator aims to minimize both the penalty for drawing power beyond a predefined limit and the customer discomfort.
Thus, the cost function $f_i$ %for aggregator $i$
is given by:
\begin{align*}
f_i(z_i) = \sum_{j=1}^{N_{c,i}} f_{ij} (z_{ij}) 
+ \eta \sum_{t=1}^{h_t} \max \left(\sum_{j=1}^{N_{c,i}} z_{ij}[t] - \tau_i, 0\right),
\end{align*}
where $f_{ij}(z_{ij}) = \sum_{t=1}^{h_t} c_{ij}(z_{ij}[t])$ with
\(c_{ij}(z_{ij}[t]) = \frac{1}{z_{ij}[t]} - \frac{1}{a_{ij}}\) if \(0 <z_{ij}[t]< a_{ij})\), and
\(c_{ij}(z_{ij}[t]) = 0\) if $z_{ij}[t] \geq a_{ij}$.
standing for the customer discomfort in cluster $j$, $\tau_i$ is the upper limit of power consumption agreed upon between aggregator $i$ and the utility, and $\eta$ is the associated per-unit cost of exceeding this limit.
In real-world applications, the optimal
$z^\star_{ij}$ after solving \eqref{prob:sharing_exam} is sent to the
corresponding cluster $j$ of aggregator $i$. Then, controllers at
the cluster level will use $z_{ij}^{\star}$ as a reference to regularize the power flow in the cluster.

To apply the hADMM to \eqref{prob:sharing_exam}, we consider the power network as a tree network where the utility, aggregators, and clusters are the root, parent nodes, and leaf nodes, respectively.
\changed{
Let us define the problem variables and functions:
\begin{itemize}
\item At the root node: $x_{1,1} = [z_i]_{i=1\dots N_a}$ and $f_{1,1}(x_{1,1}) = \lambda \Vert \sum_{i=1}^{N_a} z_i- \beta \Vert_2^2$ with $z_i \geq 0$.
\item At the aggregator nodes: 
  $x_{2,i} = [z_{ij}]_{j=1\dots N_{c,i}}$ and $f_{2,i}(x_{2,i}) = \eta \sum_{t=1}^{h_t} \max \left(\sum_{j=1}^{N_{c,i}}z_{ij}[t] - \tau_i, 0\right)$.
\item At the cluster nodes: 
  $x_{3,ij} = z_{ij}$ and $f_{3,ij} = f_{ij}$.
\end{itemize}
}
The sharing problem \eqref{prob:sharing_exam} is now written in the form of \eqref{eq:prob-tree-opt}. 

For our numerical experiment, the power network has three aggregators and each aggregator is connected to two clusters.
The step size is 15 minutes and the planning period horizon is two hours ($h_t = 8$).
The other parameters are 
$\beta = [28, 16, 20, 22, 18, 24, 15, 18]$,
$\tau = [9, 8, 11]$,
$\eta = 5$, $\lambda = 1$,
the discomfort thresholds for the clusters in the three aggregators are $[a_{ij}] = [[2.0,5.0],[3.0,2.0],[3.0,5.0]]$.
\changed{
the penalty parameters for the methods 
are $\rho_\mathrm{hADMM} = 1.4$, $\rho_\mathrm{nADMM} = 1$, and $\rho_\mathrm{nADMM} = 1.5$, which were chosen by minimizing the number of iterations for each method. 
The stopping tolerance is $\varepsilon_\mathrm{abs} = 10^{-3}$.
}

Note that in this example, the cost functions are non-smooth and not strongly convex.
The convergence of the hADMM for the power sharing experiment is shown in Fig.~\ref{fig:sharing_converge}. 
Evidently, the optimality gap and the primal residual converge 
as the number of iterations increases, validating Theorem~\ref{thm:converge}.
\changed{
Furthermore, the hADMM has a significantly lower communication cost compared to the nADMM and fADMM as shown in 
Table~\ref{tab:sharing_problem}.
}

\begin{table}
\centering
\changed{
  \caption{Performance comparison for the power sharing problem.}
  \label{tab:sharing_problem}
  \begin{tabular}{|c|c|c|c|}
    \hline
    Performance metric &nADMM &fADMM &hADMM \\
    \hline
    $\mathtt{MaxNodeIters}$ &351 &\textbf{87} &92   \\ 
    \hline
    $\mathtt{MaxNodeVals}$ &14416 &8352 &\textbf{4416}\\
    \hline
    $\mathtt{TotalVals}$ &31856 &25056 &\textbf{17664} \\
    \hline
  \end{tabular}
  }
  \vspace{-1.5em}
\end{table}

\begin{figure}[!tb]
  \centering
  \changed{
  \subfloat[Optimality gap]{\includegraphics[width=0.5\linewidth]{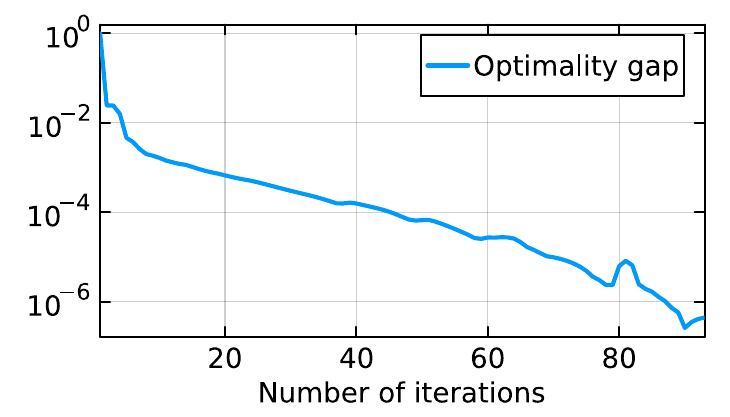}} 
  \subfloat[Primal-dual residual]{\includegraphics[width=0.5\linewidth]{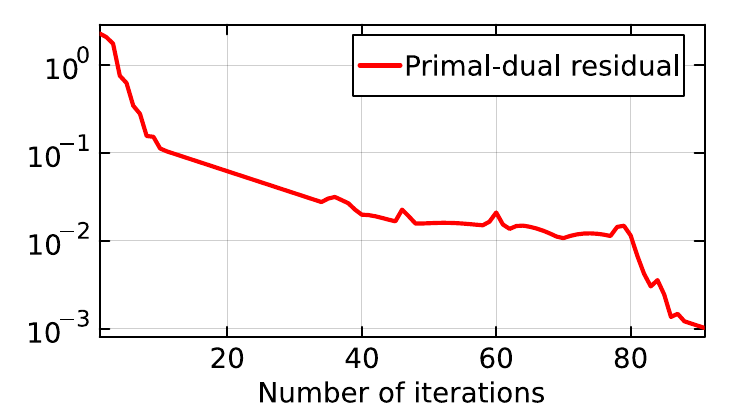}}   
  \caption{Convergence of the hADMM for the power sharing problem.
  }
  \label{fig:sharing_converge}
  \vspace{-0.5em}
  }
\end{figure}

\subsection{Convergence rate of hADMM for different topologies}
\label{sec:examples:convergence_rate}

\changed{This experiment validates the linear convergence rate analysis of the hADMM for different topologies.}
Two topologies are considered with the same numbers of levels and nodes as shown in Fig.~\ref{fig:tree_structures}.
We apply hADMM to solve the following optimal consensus problem:
\begin{subequations}
  \begin{align}
    \minimize 
    &\sum_{d=1}^{D} \sum_{i\in \calP_d \cup \calL_d}^{} f_{d,i}(x_{d,i})
      \label{eq:consensus_prob}
    \\
    \text{s.t.}\, & x_{d,i} = x_{d+1,j},\, \forall d \in \bbN_{D-1}, i \in \calP_d, j \in \calC_{d,i}\text,
  \end{align}
\end{subequations}
\changed{
in which %for $d = 1$,
$f_{1,1} = \frac{1}{2} \|x_{1,1} - u\|_2^2 + \|x_{1,1}\|_2^4$ and
for $d>1$,
\(f_{d,i}(x_{d,i}) = \frac{1}{2} \|x_{d,i} - u\|_2^2 + \lambda_{d,i} \sum_{j=1}^{n_{d,i}}\log(1 + e^{-(x_{d,i})_j}) \), where $(x_{d,i})_j$ is the $j$-th element of vector $x_{d,i}$.
The cost functions are strongly convex with moduli $\alpha_{d,i} = \frac{1}{2}$. 
The gradients of the non-root cost functions are Lipschitz continuous with constants $l_{d,i} = 1 + \frac{\lambda_{d,i}}{4}$.}
The remaining parameters are  $u = 2$, $\lambda_{d,i} = 1$, \changed{and the stopping tolerance} $\varepsilon_\mathrm{abs} = 10^{-4}$.
The same initialization of the variables is used across the topologies.
\changed{The theoretical linear convergence rate parameters given by Theorem~\ref{thm:convergence_rate} are $\gamma^\star = 0.097$ for Topology 1 and $\gamma^\star = 0.023$ for Topology 2, where the largest $\gamma^\star$ was determined using bilevel optimization as presented in Remark~\ref{rm:bilevel}.}
\changed{As shown in Fig.~\ref{fig:convergence_rate}, which plots ${\hat{V}^k}\slash{\hat{V}^0}$ over the iterations, $\hat V^k \slash \hat{V}^0$ in Topology 1 with the higher $\gamma^\star$ contracts more than three times as fast as that in Topology 2 with the lower $\gamma^{\star}$, requiring $33$ iterations versus $104$ iterations.}
\changed{Fig.~\ref{fig:convergence_rate} also shows the contraction behavior of $\hat V^k \slash \hat{V}^0$ in both topologies.
The two dashed lines illustrate the decaying upper-bounds 
corresponding to the two $\gamma^\star$ values, validating the Lyapunov contraction rates \eqref{eq:lya_ctrc}.
}

\begin{figure}[!tb]
  \centering
  \ifdefined\SubmittedPaper
  \includegraphics[width=0.75\linewidth]{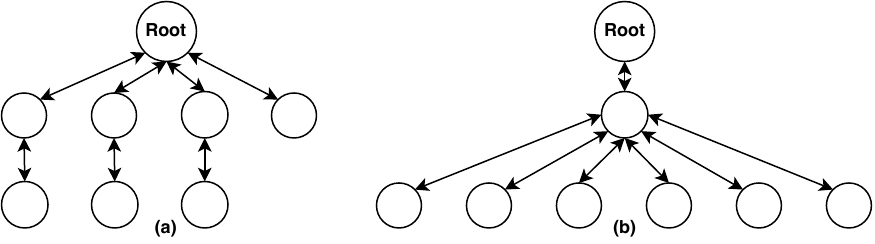}
  \else
  \includegraphics[width=0.9\linewidth]{figs/topologies.pdf}
  \fi
  \caption{
  Two topologies with $D=3, N=8$: a) Topology 1 has balance between levels 2 \& 3, and b) Topology 2 has most nodes in the lowest level.}
  \label{fig:tree_structures}
  \vspace{-0.5em}
\end{figure}

\begin{figure}[!tb]
  \centering
  \ifdefined\SubmittedPaper
  \includegraphics[width=0.75\linewidth]{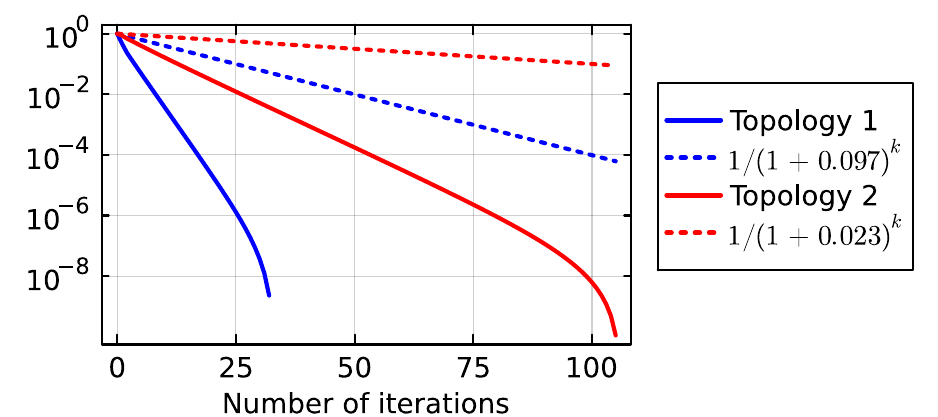}
  \else
  \includegraphics[width=0.9\linewidth]{figs/V_trajectory_hADMM.pdf}
  \fi
  \caption{\changed{Convergence of $\hat V^k/ \hat V^0$ by hADMM for the two topologies in Fig.~\ref{fig:tree_structures}.}
  }
  \label{fig:convergence_rate}
  \vspace{-0.5em}
\end{figure}

\section{Conclusion}
\label{sec:conclusion}

This paper addresses a general hierarchical optimization problem over tree-structured networks.
The key technique lies in the reformulation of the augmented Lagrangian to explicitly exploit the hierarchical network structure, which enables the proposed hierarchical ADMM (hADMM) algorithm to solve the optimization problem in a fully distributed manner and at substantially reduced communication costs.
\changed{The convergence guarantees and convergence rate of the hADMM are established under convexity assumptions.
Three numerical examples thoroughly demonstrated the effectiveness of the proposed hADMM algorithm and validated the theoretical results. 
Future work will focus on extending the hADMM framework beyond the fixed bidirectional tree setting considered here, including time-varying links, asynchronous updates, and unreliable communication with delays or packet losses.}

\appendices
\input{appendix.tex}

\label{sec:references}
\bibliographystyle{IEEEtran}
\bibliography{refs}

\end{document}

%% file: deflatex.tex
\font\myownfont=cmr17 scaled \magstep5
\def\psfancypar#1#2{\def\biginitial#1{{\myownfont#1}}%
	\def\makeinitial#1{\setbox8\hbox{\strut\vbox to 1.3ex
			{\hbox{\biginitial#1}\vskip -4pc plus 3.5pc minus 3.5pc}}}%
	\makeinitial#1%
	\ifdim\parindent>1.3\wd8\dimen8=\parindent
	\else\dimen8=1.3\wd8\fi
	\hangindent=\dimen8\hangafter=-2
	\noindent
	\strut\hskip-1\dimen8\box8{\sc#2}}%

\DeclareMathOperator*{\tridiag}{tridiag}

%% file: appendix.tex
\ifdefined\SubmittedPaper
\else
\section{The General nADMM Algorithm}
\label{app:nADMM}

Let us define recursively the aggregate variable vector $X_{d,i} \coloneq \big[x_{d,i}^\top,~ [X_{d+1,j}]_{j\in \calC_{d,i}}^\top \big]^\top$, that concatenates all the node variables in the subtree of node $(d,i)$.
For a leaf node $(d,i)$, $X_{d,i} \coloneq x_{d,i}$.
Similarly, define recursively the aggregate objective function of the subtree of node $(d,i)$ as
$F_{d,i}(X_{d,i}) = f_{d,i}(x_{d,i}) + \sum_{j\in \calC_{d,i}} F_{d+1,j}(X_{d+1,j})$, 
where $F_{d,i}(X_{d,i}) = f_{d,i}(x_{d,i})$ for a leaf node.
Finally, all the constraints~\eqref{eq:tree-cstr} in the subtree of a parent node $(d,i)$ can be aggregated 
as $\calT_{d,i} X_{d,i} = \tau_{d,i}$, which can be split recursively as $A_{d,ij} x_{d,i} + B_{d,ij} x_{d+1,j} = c_{d,ij}$ 
and $\calT_{d+1,j} X_{d+1,j} = \tau_{d+1,j}$ for all $j\in \calC_{d,i}$,
where
\begin{align*}
&\calT_{d,i} \!=\! \begin{bmatrix} [A_{d,ij}]_{j\in \calC_{d,i}} &\!\!\![B_{d,ij}~ 0]_{j\in \calC_{d,i}}^{\diag}\\0 &\![\calT_{d+1,j}]_{j\in \calC_{d,i}} 
\end{bmatrix}\!,\,
\tau_{d,i} \!=\! \begin{bmatrix} [c_{d,ij}]_{j\in \calC_{d,i}} \\ [\tau_{d+1,j}]_{j\in \calC_{d,i}} 
\end{bmatrix}\!\text.
\end{align*}
Then, \eqref{eq:prob-tree-opt} is reformulated as
\begin{subequations}
\label{eq:nADMM-X1}
\begin{align}
  \minimize_{X_{1,1}} \; &F_{1,1}(X_{1,1}) = f_{1,1}(x_{1,1}) +\!\! \sum_{j\in \calC_{1,1}} F_{2,j}(X_{2,j})\text,
                           \label{eq:nADMM-X1:obj}
  \\
  \text{s.t.} \; &A_{1,1j} x_{1,1} + B_{1,1j} x_{2,j} = c_{1,1j}, \quad\forall j\in \calC_{1,1}\text,\\
  & \calT_{2,j} X_{2,j} = \tau_{2,j}, \quad\forall j\in \calC_{1,1}\text.
  % & A_{d,ij} x_{d,i} + B_{d,ij} x_{d+1,j} = c_{d,ij}, \\ 
  %   &\qquad\forall d = 2, \dots, D-1, i \in \calP_d, j\in \calC_{d,i}.
    \label{eq:nADMM-X1:constraint}
\end{align}  
\end{subequations}
Using the standard ADMM, \eqref{eq:nADMM-X1} is solved by the %following
iterations
\begin{subequations}
\label{eq:nADMM-XX}
\begin{align} 
    x_{1,1}^{k+1} =& \displaystyle\argmin_{x_{1,1}} f_{1,1}(x_{1,1})
    + \frac{\rho}{2}\Vert A_{1,1} x_{1,1} \!-\! v_{1,1}^k \Vert_2^2,
    \\
    X_{2,j}^{k+1} =& \displaystyle \argmin_{X_{2,j}} F_{2,j}(X_{2,j}) \!+\! \frac{\rho}{2} \Vert B_{1,1j} x_{2,j} \!-\!  q_{1,1j}^{k+1} \Vert_2^2,
    \label{eq:nADMM-X2}
    \\
    &\;\quad\text{s.t.}~\calT_{2,j} X_{2,j} = \tau_{2,j}, \nonumber
    % &~~~+\! \frac{\rho}{2}\Vert A_{1,1j} x_{1,1}^{k+1} \!+\! B_{1,1j}  x_{2,j} \!-\! b_{1,1j} \!+\! \frac{1}{\rho} y_{1,1j}^k \Vert_2^2,
    \\
    y_{1,1j}^{k+1} =& y_{1,1j}^k + \rho(A_{1,1j} x_{1,1}^{k+1} + B_{1,1j} x_{2,j}^{k+1} - c_{1,1j}),
\end{align}
\end{subequations}
where %for $d=1$ and $j\in \calC_{1,1}$, 
$v_{1,1}^k \!=\! - [B_{1,i}]_{j\in \calC_{1,i}}^{\diag} [x_{2,j}^{k}]_{j\in \calC_{1,1}} \!+\! c_{1,1} \!-\! \rho^{-1} y_{1,1}^k$ and
$q_{1,1j}^{k+1} \!= \!-\! A_{1,1j} x_{1,i}^{k+1} \!+\! c_{1,1j} \!-\! \rho^{-1} y_{1,1j}^k$.

For each node $(2,j)$ at level $d=2$, let us define
\begin{equation*}
  \hat f_{2,j}(x_{2,j}; q) =  f_{2,j}(x_{2,j}) + \frac{\rho}{2} \Vert B_{1,1j} x_{2,j} -  q \Vert_2^2\text.
\end{equation*}
% where $(d,i)$ is the parent node of $(d+1,j)$. %, $i \in \calP_{d,j}$.
%
If $(2,j)$ is a leaf node, \eqref{eq:nADMM-X2} becomes the proximal operator $\argmin_{x_{2,j}} \hat f_{2,j}(x_{2,j}; q_{1,1j}^{k+1})$.
If $(2,j)$ is a parent node, \eqref{eq:nADMM-X2} is equivalent to
\begin{subequations}
\label{op:nADMM-X2}
\begin{align}
    \minimize_{X_{2,j}} \;& \hat f_{2,j}(x_{2,j}; q_{1,1j}^{k+1}) + \textstyle\sum_{l \in C_{2,j}}^{} F_{3,l}(X_{3,l})\text,
    \\
    \text{s.t.}\;& A_{2,jl} x_{2,j} + B_{2,jl} x_{3,l} = c_{2,jl}\text, \quad \forall l \in \calC_{2,j}\text,
    \\
    & \calT_{3,l} X_{3,l} = \tau_{3,l}\text, \; \forall l \in \calC_{2,j}\text.
\end{align}
\end{subequations}
Observe that \eqref{op:nADMM-X2} has the same form as \eqref{eq:nADMM-X1}, hence it can be solved using the same algorithm in a recursive manner.
This recursive, or nested, nADMM algorithm is formalized in Algorithm~\ref{alg:nADMM},
where for $d=2,\dots, D-1$ and $j\in \calC_{d,i}$,
\begin{subequations}
    \begin{align}
        &v_{d,i}^k \!=\! - [B_{d,i}]_{j\in \calC_{d,i}}^{\diag} [x_{d+1,j}^{k}]_{j\in \calC_{d,i}} \!+\! c_{d,i} \!-\! \rho^{-1} y_{d,i}^k, 
        \label{eq:nADMM-v}
        \\
        &q_{d,ij}^{k+1} \!= \!-\! A_{d,ij} x_{d,i}^{k+1} \!+\! c_{d,ij} \!-\! \rho^{-1} y_{d,ij}^k,
        \label{eq:nADMM-q}
        \\
        &\hat f_{d+1,j}(x_{d+1,j}; q) =  f_{d+1,j}(x_{d+1,j}) \nonumber \\
        &\qquad\qquad\qquad\qquad\qquad + \frac{\rho}{2} \Vert B_{d,ij} x_{d+1,j} - q \Vert_2^2\text.
        \label{eq:nADMM-hatf}
    \end{align}
\end{subequations}

\begin{algorithm}[t]
  \caption{Nested ADMM (nADMM) algorithm for hierarchical optimization problem \eqref{eq:prob-tree-opt}}
   \label{alg:nADMM}
  \begin{algorithmic}[1]
    % \Require Tree network, $\lambda = \frac{1}{\rho} > 0$.
    % \Ensure Reaching to an optimal solution.
    \Function{nADMM}{$(d,i), q$}
      \If{$(d,i)$ is leaf node}
        \State {\bf return} 
        $\argmin_{x} \hat f_{d,i}(x; q)$ \Comment{with $\hat f_{d,i}$ in \eqref{eq:nADMM-hatf}}
      \EndIf
      % \State $k = 1$
      \Repeat{} with counter $k$ starting from $1$
        \State Compute $v_{d,i}^k$ in \eqref{eq:nADMM-v}
        \If{$(d,i)$ is root node}
          \State $x_{d,i}^{k+1} = \displaystyle \argmin_{x} f_{d,i}(x) \!+\! \frac{\rho}{2}\Vert A_{d,i} x  \!-\! v_{d,i}^k\Vert_2^2$
        \Else
          \State $x_{d,i}^{k+1} = \displaystyle \argmin_{x} \hat f_{d,i}(x_{d,i}; q) \!+\! \frac{\rho}{2}\Vert A_{d,i} x  \!-\! v_{d,i}^k\Vert_2^2$
        \EndIf
        \For{$j \in \calC_{d,i}$} \Comment{In parallel for all $j$}
        %   \State $q^{k+1}_{d,ij} = - A_{d,ij} x_{d,i}^{k+1} + c_{d,ij} - \frac{1}{\rho} y_{d,ij}^k$
          \State Compute $q^{k+1}_{d,ij}$ in \eqref{eq:nADMM-q}
          \State $x^{k+1}_{d+1,j} =$ \Call{nADMM}{$(d+1,j)$, $q^{k+1}_{d,ij}$}
        \EndFor
        \State $y_{d,ij}^{k+1} \!=\! y_{d,ij}^k \!+\! \rho(A_{d,ij} x_{d,i}^{k+1} \!+\! B_{d,ij} x_{d+1,j}^{k+1} \!-\! c_{d,ij})$
         % \Until{ $\max_{j\in \calC_{d,i}} \Vert A_{d,i}x_{d,i}^{k+1} - B_{d,ij}x_{d+1,j}^{k+1} \Vert_\infty \leq \varepsilon $.}
      \Until{convergence}
      \State {\bf return} %$x_{d,i}^{\star} =
      $x_{d,i}^{k+1}$
    \EndFunction
    
    \State \Call{nADMM}{$(1,1)$, $[]$} \Comment{Solve problem \eqref{eq:prob-tree-opt} at root node}
% \vspace{-1em}
\end{algorithmic}
\end{algorithm}
\fi

% \ifdefined\SubmittedPaper
% \else
\section{Proof of Theorem \ref{thm:converge}\label{apdx:thm1_prf}}
% \changed{
% \begin{definition} \label{def:cluster}
% $\bar a$ is a cluster point of  $\{a_k\}$  if, for given  $\epsilon \in \bbR_{>0}$  and given  $N>0$,  $\exists n \in \bbN_{>0}$  such that  $\Vert a_n - \bar a\Vert_2 \leq \epsilon$.
% \end{definition}
% }

\begin{lemma}[\cite{boyd2004convex}]
\label{lem:convex}
For any convex function $f$ %, the following inequality holds $\forall z , z^\prime \in \mathrm{Dom}(f)$,
and for all $z, z^\prime \in \mathrm{Dom}(f)$,
$\left( \partial f(z) - \partial f(z^\prime) \right)^\top(z - z^\prime) \geq 0$.
\end{lemma}
To simplify the presentation, let us denote:
\begin{align}
&r_d^{k+1} = A_d x_d^{k+1} + B_d x_{d+1}^{k+1} - c_d,\quad d=1,\dots,D-1,
\label{eq:conv-proof:res}
% \\
% &\changed{s_d^{k+1} = \rho A_d^\top B_d (x_{d+1}^{k+1} - x_{d+1}^k),\quad d=1,\dots,D-1}\text, 
\\
&\changed{\tilde{y}_d^k = y_d^k - y_d^\star,~ \tilde{x}_{d+1}^k = B_d (x_{d+1}^k - x_{d+1}^\star)}\text,
\\
&\changed{\delta_d^{k+1} = \rho B_{d}(x_{d+1}^{k+1} - x_{d+1}^k),\quad d=1,\dots,D-1}\text,
\\
&J^k = \sum_{i=1}^{D} f_d(x_d^k), \quad J^\star = \sum_{i=1}^{D} f_d(x_d^\star),
\label{eq:conv-proof:J}
\\
&\changed{V_d^k = \frac{1}{\rho}\Vert \tilde{y}_d^k\Vert_2^2 + \rho\Vert \tilde{x}_{d+1}^k \Vert_2^2 \text,~
V^k = \sum_{d=1}^{D-1} V_d^k \text.} \label{eq:conv-proof:Vk}
\end{align}
Here, $([ x_d^{\star}]_{d\in \bbN_{D}}, [ y_d^{\star}]_{d\in \bbN_{D-1}})$ is any KKT point satisfying \eqref{thm:prf_KKT1}-\eqref{thm:prf_KKT4}.
We will prove the theorem in two steps.
Inspired by the proof of ADMM in \cite{goldstein2014fast, nishihara2015general}, we first prove that
\begin{align} \label{eq:conv-proof:prf_increment}
V^k \!\geq\!  V^{k+1} \!+\! \rho \sum_{d=1}^{D-1} \Vert r_d^{k+1} \!-\! B_d (x_{d+1}^{k+1} \!-\! x_{d+1}^k)\Vert_2^2.
\end{align}
After the forward pass \eqref{eq:forward}, $x_d^{k+1}$ satisfies the generalized first-order condition \cite[Chapter 3.7]{beck2017first}: 
$0 \in \partial f_1 (x_1^{k+1}) \!+\! A_1^\top (y_1^k \!+\! \rho  (A_1 x_1^{k+1} \!+\! B_1 x_2^k \!-\! c_1))$, 
$0 \in \partial f_D (x_D^{k+1}) +\! B_{D-1}^\top (y_{D-1}^k \rho (A_{D-1} x_{D-1}^{k+1} \!+\! B_{D-1} x_D^{k+1} \!-\! c_{D-1}))$, 
and $0 \in \partial f_d (x_d^{k+1}) + A_d^\top (y_d^k + \rho  (A_d x_d^{k+1} \!+\! B_d x_{d+1}^k - c_d)) + B_{d-1}^\top (y_{d-1}^k + \rho(A_{d-1} x_{d-1}^{k+1} + B_{d-1} x_d^{k+1} - c_{d-1}))$
% \begin{align*} 
%     &0 \in \partial f_1 (x_1^{k+1}) \!+\! A_1^\top (y_1^k \!+\! \rho  (A_1 x_1^{k+1} \!+\! B_1 x_2^k \!-\! c_1)),  
%     % &\text{if~} d=1,
%     \\
%     &0 \in \partial f_D (x_D^{k+1}) +\! B_{D-1}^\top (y_{D-1}^k 
%     \\
%     &~~~~~~ +  \rho (A_{D-1} x_{D-1}^{k+1} \!+\! B_{D-1} x_D^{k+1} \!-\! c_{D-1})),
%     % &\text{if~} d=D,
%     \\
%     &0 \in \partial f_d (x_d^{k+1}) \!+\! A_d^\top (y_d^k \!+\! \rho  (A_d x_d^{k+1} \!+\! B_d x_{d+1}^k \!-\! c_d)) 
%     \\
%     &~~~~~~ +
%     B_{d-1}^\top (y_{d-1}^k \!+\! \rho(A_{d-1} x_{d-1}^{k+1} \!+\! B_{d-1} x_d^{k+1} \!-\! c_{d-1})),
%     % &\text{elsewhere.}
% \end{align*}
for $1< d <D$.
Substituting \eqref{eq:conv-proof:res} into these conditions and noting that $y^{k+1}_d = y^k_d + \rho r^{k+1}_d$ from the backward pass \eqref{eq:backward}, we have
\begin{subequations}
\begin{align} \label{thm:prf_KKTyk1}
    &0 \in \partial f_1 (x_1^{k+1}) + A_1^\top y_1^{k+1} - A_1^\top \delta_1^{k+1},
    \\
    &0 \in \partial f_D (x_D^{k+1}) + B_{D-1}^\top y_{D-1}^{k+1},
    \label{thm:prf_KKTyk2}
    \\
    &0 \in \partial f_d (x_d^{k+1}) + A_d^\top y_d^{k+1} - \changed{A_d^\top \delta_d^{k+1}} + B_{d-1}^\top y_{d-1}^{k+1},
    \label{thm:prf_KKTyk3}
\end{align}
\end{subequations}
for $1 < d < D$.
%
% \TN{In this part, the term $\rho (x_{d+1}^{k+1} - x_{d+1}^k)^\top B_d^\top A_d$ is essentially the residual $s$ defined and used in the stopping criterion in section~\ref{sec:stop-crit}.  We should reuse it here, also to shorten the math.}
%
Because of the convexity of $f_d$, $x_d^{k+1}$ minimizes the following functions:
$f_1(x) +  (A_1 y_1^{k+1} - A_1^\top \delta_1^{k+1})^\top x$,
$f_D(x) +  \big(y_{D-1}^{k+1}\big)^\top B_{D-1} x$, and
$f_d(x) + \big(A_d^\top y_d^{k+1} - A_d^\top \delta_d^{k+1} + B_{d-1}^\top y_{d-1}^{k+1}\big)^\top x$
for $1<d<D$.
% Thus, the following inequalities hold
Thus, it has
\begin{align*} 
    &f_d(x_d^{k+1})  + \Theta_d^{k+1} x_d^{k+1} \leq f_d(x_d^\star) + \Theta_d^{k+1} x_d^\star,
\end{align*}
where $\Theta_1^{k+1} = \big( A_1^\top y_1^{k+1} - A_1^\top \delta_1^{k+1} \big)^\top, \Theta_D^{k+1} = \big( y_{D-1}^{k+1} \big)^\top B_{D-1}$, and 
$\Theta_d^{k+1} = (A_d^\top y_d^{k+1} - A_d^\top \delta_d^{k+1} + B_{d-1}^\top y_{d-1}^{k+1})^\top A_d$ for $1 < d < D$. 
Summing the above inequalities for all $d = 1,\dots,D$ and using \eqref{eq:conv-proof:J}, we have
\begin{equation}
    \label{eq:upper}
    J^{k+1} \leq  J^{\star} + g^{k+1},
\end{equation}
where $g^{k+1} = \sum_{d=1}^D (\Theta_d^{k+1})^\top (x_d^\star \!-\! x_d^{k+1})$ is rewritten as
\begin{align}
  g^{k+1} &=  (y_{D-1}^{k+1})^\top B_{D-1} (x_D^\star \!-\! x_D^{k+1}) \!+\! (y_1^{k+1})^\top A_1 (x_1^\star \!-\! x_1^{k+1}) \nonumber \\
  &\quad + \sum_{d=1}^{D-1} (\delta_d^{k+1})^\top A_d (x_d^{k+1} \!-\! x_d^\star) \nonumber \\ 
  &\quad + \sum_{d=2}^{D-1} (A_d^\top y_d^{k+1} + B_{d-1}^\top y_{d-1}^{k+1})^\top (x_d^\star - x_d^{k+1})\text. \label{eq:g}
\end{align}
By the KKT condition~\eqref{thm:prf_KKT4} and by \eqref{eq:conv-proof:res}, we can derive
\begin{equation}
  \label{eq:conv-proof:res-alt}
  A_d (x_d^{k+1} - x_d^\star) +  \tilde{x}_{d+1}^{k+1} =  r_d^{k+1} \text.
\end{equation}
Left-multiplying both sides by $(y_d^{k+1})^\top$ then summing the equations for all $d$ from $1$ to $D-2$, noting that
$-(y_{D-1}^{k+1})^\top \tilde{x}_D^{k+1} = (y_{D-1}^{k+1})^\top (-r_{D-1}^{k+1} - A_{D-1} (x_{D-1}^\star - x_{D-1}^{k+1}))$
from \eqref{eq:conv-proof:res-alt}, we can rewrite \eqref{eq:g} as
\begin{multline}
  g^{k+1} = \!\sum_{d=1}^{D-1} \! \Big( \!(\delta_d^{k+1})^\top A_d (x_d^{k+1} - x_d^\star) - (y_d^{k+1})^\top r_d^{k+1} \!\Big) \text.
  \label{eq:conv-proof:g-final}
\end{multline}
Since $(x^\star, y^\star)$ is a saddle point for the augmented Lagrangian $L$ at any $\rho$ (see \cite[Sec~5.4]{boyd2004convex}), from \eqref{def:Lagrangian} and \eqref{eq:conv-proof:res}, we have
\(
    J^\star = L(x^\star, y^\star)\Big|_{\rho = 0}  
    \leq  
    L(x^{k+1}, y^\star)\Big|_{\rho = 0} 
    = J^{k+1} + \sum_{d=1}^{D-1} \big(y_d^{\star} \big)^\top r_d^{k+1}
\).
% \begin{align*} 
%     J^\star = L(x^\star, y^\star)\Big|_{\rho = 0}  
%     &\leq  
%     L(x^{k+1}, y^\star)\Big|_{\rho = 0} 
%     \nonumber \\
%     &= J^{k+1} \!+\! \sum_{d=1}^{D-1} \big(y_d^{\star} \big)^\top r_d^{k+1}\text.
% \end{align*}
%
Then, combining with \eqref{eq:upper}, we obtain
\begin{align}
    -\sum_{i=0}^{D-1} \big(y_d^{\star} \big)^\top r_d^{k+1} \leq  J^{k+1} - J^\star \leq g^{k+1}.
    \label{eq:conv-proof:bounds}
\end{align}
% From \eqref{eq:conv-proof:bounds}, from $J^{k+1} \rightarrow J^\star$ if $r_d^{k+1} \rightarrow 0$ and $B_d(x_{d+1}^{k+1} - x_{d+1}^k) \rightarrow 0$ no matter what $x_d^\star$ and $y_d^\star$ are.
This implies $0 \leq 2 \big( g^{k+1} + \sum_{d=1}^{D-1} y_d^{\star\top} r_d^{k+1} \big)$, which from \eqref{eq:conv-proof:g-final}
\begin{multline}
0 \leq \sum_{d=1}^{D-1} \!\!\Big( 2 (\delta_d^{k+1})^\top A_d (x_d^{k+1} \!-\! x_d^\star) \!-\! 2 (\tilde{y}_d^{k+1} )^\top r_d^{k+1} \Big)\text.
  \label{eq:gyr}
\end{multline}
We will rewrite the right hand side (RHS) of \eqref{eq:gyr} in terms of $V^k$ and $V^{k+1}$.
First, using \eqref{eq:conv-proof:res-alt}, the first term in the summand can be expressed as
\begin{align*}
	&2 (\delta_d^{k+1})^\top A_d   (x_d^{k+1} - x_d^\star) =  2 (\delta_d^{k+1})^\top (r_d^{k+1} - \tilde{x}_{d+1}^{k+1}) \\
	% &=  2 (\delta_d^{k+1})^\top r_d^{k+1} - 2 (\delta_d^{k+1})^\top \tilde{x}_{d+1}^{k+1} \\
	&=  2 (\delta_d^{k+1})^\top r_d^{k+1}  
    - 2 \frac{1}{\rho} \Vert \delta_d^{k+1} \Vert_2^2 + 2 \rho (\tilde{x}_{d+1}^k)^\top  \tilde{x}_{d+1}^{k+1} \\
	% ={} & \rho \Big( 2 (x_{d+1}^{k+1} - x_{d+1}^k)^\top B_d^{\top} r_d^{k+1} - \Vert B_d (x_{d+1}^{k+1} - x_{d+1}^k) \Vert_2^2 \\
    % &\quad + \Vert B_d(x_{d+1}^k - x_{d+1}^\star) \Vert_2^2 - \Vert B_d (x_{d+1}^{k+1} - x_{d+1}^\star) \Vert_2^2 \Big) \\
	&= \!-\! \frac{1}{\rho}\Vert \rho r_d^{k+1} \!-\! \delta_d^{k+1} \Vert_2^2 \!+\! \rho \Vert r_d^{k+1}\Vert_2^2  + \rho \Vert \tilde{x}_{d+1}^k \Vert_2^2 - \rho \Vert \tilde{x}_{d+1}^{k+1} \Vert_2^2\text.
\end{align*}
Using $y_d^{k+1} = y_d^k + \rho r_d^{k+1}$, the second term in the summand of \eqref{eq:gyr} can be rewritten as:
\(	2(\tilde{y}_d^{k+1})^\top r_d^{k+1} 
    = 2(\tilde{y}_d^k)^\top r_d^{k+1} + 2\rho \Vert r_d^{k+1} \Vert_2^2
	% ={} &\frac{2}{\rho}(y_d^k-y_d^{\star})^\top (y_d^{k+1} - y_d^k) + 2 \rho \Vert r_d^{k+1} \Vert_2^2 \\
	% ={} & -\frac{2}{\rho}(y_d^k-y_d^{\star})^\top (y_d^{\star} - y_d^{k+1}) - \frac{2}{\rho}\Vert y_d^k - y_d^{\star}\Vert_2^2 + 2\rho \Vert r_d^{k+1} \Vert_2^2 \\
    = \frac{1}{\rho} \Vert \tilde{y}_d^{k+1}\Vert_2^2 - \frac{1}{\rho}\Vert \tilde{y}_d^k \Vert_2^2 + \rho \Vert r_d^{k+1} \Vert_2^2 \text.
\)
Using these expressions and \eqref{eq:conv-proof:Vk}, the summand in \eqref{eq:gyr} is rewritten as
\begin{multline}
	% 2\rho(x_{d+1}^{k+1} - x_{d+1}^k)^\top B_d^\top A_d (x_d^{k+1} \!-\! x_d^\star) 
    2 \big(\delta_d^{k+1}\big)^\top A_d (x_d^{k+1} \!-\! x_d^\star)
    - 2(\tilde{y}_d^{k+1})^\top r_d^{k+1} \\
	= V^k_d - V^{k+1}_d - \frac{1}{\rho} \Vert \rho r_d^{k+1} - \delta_d^{k+1}\Vert_2^2\text.
    \label{eq:conv-proof:prf_expr}
\end{multline}
% where $V_d^k = \frac{1}{\rho}\Vert y_d^k - y_d^{\star}\Vert_2^2 + \rho\Vert B_d(x_{d+1}^k - x_{d+1}^\star) \Vert_2^2$. 
Summing \eqref{eq:conv-proof:prf_expr} for all $d$ from $1$ to $D-1$ and using \eqref{eq:conv-proof:Vk} and \eqref{eq:gyr}, we have \eqref{eq:conv-proof:prf_increment}.

\changed{According to \eqref{thm:prf_KKTyk2}, $x_D^{k+1}$ and $x_D^k$ are minimizers of
$f_D(x) + (y_{D-1}^{k+1})^\top B_{D-1} x$ and $f_D(x) + (y_{D-1}^k)^\top B_{D-1} x$, respectively.
Thus,
$f_D(x_D^{k+1}) - f_D(x_D^k) + \frac{1}{\rho} (y_{D-1}^{k+1})^\top \delta_{D-1}^{k+1} \leq 0$, and
$f_D(x_D^k)- f_D(x_D^{k+1}) - \frac{1}{\rho} (y_{D-1}^k)^\top  \delta_{D-1}^{k+1} \leq 0$.
Adding the inequalities leads to $(r_{D-1}^{k+1})^\top  \delta_{D-1}^{k+1} \leq 0$.
This implies
\begin{math}
\| r_{D-1}^{k+1} - \frac{1}{\rho} \delta_{D-1}^{k+1} \|_2^2
\geq
\Vert r_{D-1}^{k+1} \Vert_2^2
+ \Vert \frac{1}{\rho} \delta_{D-1}^{k+1} \Vert_2^2\text.
\end{math}
Then, from \eqref{eq:conv-proof:prf_increment}, 
% we have
\begin{multline}
  \label{eq:conv-proof:V-increment-separate}
  V^k \geq  V^{k+1} + \frac{1}{\rho} \sum_{d=1}^{D-2} \Vert \rho r_d^{k+1} - \delta_d^{k+1}\Vert_2^2 \\
  + \rho \Vert r_{D-1}^{k+1} \Vert_2^2 + \frac{1}{\rho} \Vert \delta_{D-1}^{k+1} \Vert_2^2\text.
\end{multline}
Since $V^k$ is non-increasing and non-negative, it converges and the non-negative terms in \eqref{eq:conv-proof:V-increment-separate} vanish. Hence
\begin{align}
  &r_{D-1}^{k+1} \to 0, ~ \delta_{D-1}^{k+1} \to 0, \text{\,and\,\,}r_d^{k+1}- \frac{1}{\rho} \delta_d^{k+1}\to 0,  \label{eq:conv-proof:last-layer-limits}
\end{align}
for $d=1,\dots,D-2$. Because $B_{D-1}$ has full column rank, \eqref{eq:conv-proof:last-layer-limits} gives $x_D^{k+1}-x_D^k\to 0$.
Repeating the same monotonicity argument recursively for $d=D-2,\dots,1$ gives
\begin{align}
  r_d^{k+1}\to 0,\quad x_{d+1}^{k+1}-x_{d+1}^k\to 0,\quad d=1,\dots,D-1.
  \label{eq:conv-proof:residual-increment-limits}
\end{align}}
% We emphasize that \eqref{eq:conv-proof:residual-increment-limits} alone does not imply convergence of the full sequence.
% We therefore complete the proof by a cluster-point argument.
% }
\changed{From  \eqref{eq:conv-proof:Vk}, %the sequence $\{y_d^k\}$ is bounded for all $d\in\bbN_{D-1}$, and $\{B_d(x_{d+1}^k-x_{d+1}^\star)\}$ is 
the sequences $\{y_d^k\}$ and $\{B_d(x_{d+1}^k-x_{d+1}^\star)\}$ are
bounded for all $d\in\bbN_{D-1}$.
Since each $B_d$ has full column rank, $\{x_d^k\}$ is bounded for all $d=2,\dots,D$.
Moreover, $r_1^k=A_1x_1^k+B_1x_2^k-c_1\to0$ and $\{x_2^k\}$ is bounded, so $\{A_1x_1^k\}$ is bounded.
The full-column-rank condition on $A_1$ then implies that $\{x_1^k\}$ is bounded.
Thus the full primal-dual sequence is bounded and has at least one cluster point according to the Bolzano-Weierstrass theorem.}

\changed{Let $([ \bar x_d]_{d\in\bbN_D},[\bar y_d]_{d\in\bbN_{D-1}})$ be any cluster point of the sequence.
Passing to a convergent subsequence $\{r_d^k\}$ in \eqref{eq:conv-proof:residual-increment-limits} gives $A_d\bar x_d+B_d\bar x_{d+1}=c_d$ for all $d\in\bbN_{D-1}$.
In addition, the perturbation terms $B_d(x_{d+1}^{k+1}-x_{d+1}^k)$ in
\eqref{thm:prf_KKTyk1} and \eqref{thm:prf_KKTyk3} vanish.
Using the closedness of the graph of the subdifferential for closed proper convex functions \cite[Theorem 24.4]{rockafellar1997convex}, the limit of \eqref{thm:prf_KKTyk1}-\eqref{thm:prf_KKTyk3} satisfies \eqref{thm:prf_KKT1}-\eqref{thm:prf_KKT3}.
Therefore every cluster point satisfies all KKT conditions \eqref{thm:prf_KKT1}-\eqref{thm:prf_KKT4}, and hence belongs to $\calO^\star$.
If $\calO^\star$ is a singleton, the bounded sequence has only one cluster point;
consequently, the whole sequence converges to the unique KKT point in $\calO^\star$.}
% \fi

\ifdefined\SubmittedPaper
\else
\section{Proof of Theorem \ref{thm:convergence_strcon}}
\label{apdx:thm2_prf}

\changed{Using the same definitions in \eqref{eq:conv-proof:res} to \eqref{eq:conv-proof:Vk} in Appendix~\ref{apdx:thm1_prf}, we will show that
\begin{multline}
    \changed{V^k - V^{k+1}
    \geq  \frac{1}{\rho} \sum_{d=1}^{D-1} \Vert \rho r_d^{k+1} - \delta_d^{k+1}\Vert_2^2} \\
    \changed{+ 2 \sum_{d=1}^{D} \alpha_d \Vert x_d^{k+1} - x_d^\star  \Vert_2^2\text.}
    \label{eq:prf_strform}
\end{multline} 
If \eqref{eq:prf_strform} holds, then $x_d^{k+1} - x_d^\star \rightarrow 0$ for $d=1,\dots,D$ and
$r_d^{k+1} - B_d (x_{d+1}^{k+1} - x_{d+1}^k) \rightarrow 0$ for $d=1,\dots,D-1$ as $k \to \infty$.}

To prove \eqref{eq:prf_strform}, let us consider the following inequality obtained by the strong convexity of $f_d$ in Scenario \ref{scn:str_convex}:
\begin{multline}
    J^{k+1} - J^\star = \sum_{d=1}^{D} \left( f_d(x_d^{k+1}) - f_d(x_d^\star) \right) \\
    \geq  \sum_{d=1}^{D} \partial f_d^\top(x_d^\star) (x_d^{k+1} - x_d^\star) 
      + \sum_{\changed{d=1}}^{D} \alpha_d \Vert x_d^{k+1} - x_d^\star  \Vert_2^2\text.
      \label{eq:prf_convex}
\end{multline}
Recalling $g^{k+1}$ defined in \eqref{eq:g} then using %\eqref{eq:gyr} and
\eqref{eq:conv-proof:prf_expr}, we have % in Appendix~\ref{apdx:thm1_prf},
\begin{multline}
  \changed{g^{k+1} + \sum_{d=1}^{D-1} y_d^{\star\top} r_d^{k+1} = \frac{1}{2} \left( V^k - V^{k+1} \right)}\\
  - \frac{1}{2\rho} \sum_{d=1}^{D-1} \Vert \rho r_d^{k+1} \!-\! \delta_d^{k+1})\Vert_2^2 \text.
  \label{eq:prf_gsum}
\end{multline}
On the other hand, according to \eqref{eq:conv-proof:bounds}, we have
\begin{math}
  % \label{thm:prf_gr}
  g^{k+1}  \geq J^{k+1} - J^\star\text.
\end{math}
Applying \eqref{eq:prf_convex} and \eqref{eq:prf_gsum} into this inequality leads to
% \begin{align}
%     g^{k+1}  + \sum_{d=1}^{D-1} y_d^{\star\top} r_d^{k+1} \geq  \sum_{d=1}^{D} \partial f_d^\top(x_d^\star) (x_d^{k+1} -x_d^\star) + \frac{\alpha}{2} 
%     \Vert x_d^{k+1} - x_d^\star  \Vert_2^2
%     +   \sum_{d=1}^{D-1} y_d^{\star\top} r_d^{k+1}
% \end{align}
% or
\begin{align}
  \frac{1}{2} (V^k - V^{k+1}) \geq & \sum_{d=1}^{D-1} y_d^{\star\top} r_d^{k+1} +  \sum_{d=1}^{D} \partial f_d^\top(x_d^\star) (x_d^{k+1} - x_d^\star) \nonumber \\
  &+ \frac{1}{2\rho} \sum_{d=1}^{D-1} \Vert \rho r_d^{k+1} \!-\! \delta_d^{k+1})\Vert_2^2 \nonumber \\
  &+ \sum_{\changed{d=1}}^{D} \alpha_d \Vert x_d^{k+1} - x_d^\star  \Vert_2^2\text.
  \label{eq:prf_Vk1-Vk}
\end{align}
From \eqref{eq:prf-KKTs}, the first two terms in the RHS of \eqref{eq:prf_Vk1-Vk} reduce to
\begin{align*}
  &
    \begin{multlined}[t]
      \sum_{d=1}^{D-1} \left( y_d^{\star\top} r_d^{k+1} + \partial f_d^\top(x_d^\star) (x_d^{k+1} - x_d^\star)  \right) \\
      + \partial f_D^\top(x_D^\star)(x_D^{k+1} - x_D^\star)  
    \end{multlined}
  \\
  ={}&
       \begin{multlined}[t]
         \sum_{d=1}^{D-1} \left( y_d^{\star\top} r_d^{k+1} - y_d^{\star \top} A_d (x_d^{k+1} - x_d^\star) \right) 
         - \sum_{d=1}^{D-1} y_d^{\star \top} \tilde{x}_{d+1}^{k+1} % B_d^\top(x_{d+1}^{k+1} - x_{d+1}^\star)
       \end{multlined}
  \\
  ={}&
       \begin{multlined}[t]
         \sum_{d=1}^{D-1} \left( y_d^{\star\top} (B_d x_{d+1}^{k+1} - c_d) +  y_d^{\star\top} A_d x_d^\star \right) 
         - \sum_{d=1}^{D-1} y_d^{\star \top}  \tilde{x}_{d+1}^{k+1} %% B_d^\top(x_{d+1}^{k+1} - x_{d+1}^\star)
       \end{multlined}
  \\
  ={}&
       \begin{multlined}[t]
         \sum_{d=1}^{D-1} \left( y_d^{\star\top} B_d x_{d+1}^{k+1} - y_d^{\star\top} B_d x_{d+1}^\star \right)
         - \sum_{d=1}^{D-1} y_d^{\star \top} \tilde{x}_{d+1}^{k+1} %B_d^\top(x_{d+1}^{k+1} - x_{d+1}^\star)
       \end{multlined}
  \\
  ={}& 0 \text.
\end{align*}
Therefore, from \eqref{eq:prf_Vk1-Vk}, \eqref{eq:prf_strform} is proved.
\changed{
By \cite[Theorem 2]{Ye2021}, under Slater’s condition and strong convexity of the objective functions,
% the optimization problem
\eqref{eq:prob-tree-opt} admits a unique minimizer in Scenario \ref{scn:str_convex} or Scenario \ref{ass:Lipschitz}.
Hence, $[x_d^\star]_{d\in \bbN_{D}}$ is the unique minimizer.
}
\fi

% \ifdefined\SubmittedPaper
% \section{Matrices in Theorem \ref{thm:convergence_rate} \label{apdx:thm3_prf}}
% \else
\section{Proof of Theorem \ref{thm:convergence_rate}
\label{apdx:thm3_prf}}
% \fi

\changed{
Let $s=\sum_{d=1}^{D-1}s_d$.  The matrix $\Xi = \Xi^\top \in\mathbb R^{4s\times4s}$ is defined as follows, where $(\ast)$ denotes the block
inferred by symmetry and 
% $0$ denotes a zero matrix of appropriate dimensions:
$0_{n\times m}$ denotes a zero matrix in $\bbR^{n\times m}$:
\begin{align*}
  &\Xi = \left[ \begin{array}{c|c:c:c} \Xi^{(1,1)} &0_{s \times s} &0_{s \times s} &(*) \\ \hline 0_{s \times s} &\Xi^{(2,2)} &0_{s \times s} &(*)
    \\ 0_{s \times s} &0_{s \times s} & \Xi^{(3,3)} &(*)
    \\ \Xi^{(4,1)} &\Xi^{(4,2)} & \Xi^{(4,3)} & \Xi^{(4,4)}\\ \end{array} \right]\text,
    \\
% \end{equation*}
% where
%\begin{subequations}
    % \begin{align*}
      &\Xi^{(1,1)} = 2 (1-t) \tridiag \Bigl( \left\{\tfrac{1}{l_{i+1}}B_i A_{i+1}^\top\right\}_{i=1,\dots, D-2}, \nonumber \\
      &\qquad\qquad \left\{\tfrac{1}{l_2}B_1 B_1^\top, \left\{\tfrac{1}{l_i}A_i A_i^\top \!+\! \tfrac{1}{l_{i+1}} B_i B_i^\top \right\}_{i=2,\dots, D-1} \right\}, \nonumber\\
      &\qquad\qquad \left\{ \changed{\tfrac{1}{l_{i+1}}A_{i+1} B_i^\top} \right\}_{i=1,\dots, D-2} \Bigr) \text, %\label{thm:prf10}
      \\
      &\Xi^{(2,2)} = \left[ \frac{2t\alpha_{d+1}}{\Vert B_d \Vert_2^2}   I_{s_d}  \right]_{d\in\bbN_{D-1}}^{\diag},\quad \Xi^{(3,3)} = \rho I_s\text, %\label{thm:prf11}
      \\
      &
      \Xi^{(4,1)} = 2(1-t) \tridiag \Bigl( 
      \{0_{s_i\times s_{i+1}}\}_{i=1,\dots,D-2}, \nonumber\\
      &\qquad\qquad\qquad\quad \Bigl\{\changed{0_{s_1\times s_1}}, \Bigl\{\changed{-\tfrac{1}{l_{i+1}} A_{i+1} A_{i+1}^\top} \Bigr\}_{i=1,\dots,D-2} \Bigr\}, \nonumber\\
      &\qquad\qquad\qquad\quad \Bigl\{-\tfrac{1}{l_{i+1}} A_{i+1} B_i^\top \Bigr\}_{i=1,\dots,D-2}
      \Bigr)\text,
    \\
    &\Xi^{(4,2)} = -\frac{\sigma}{\rho} I_s,\quad \Xi^{(4,3)} = -I_s\text,
    \\
    &\Xi^{(4,4)} = \left(\frac{1}{\rho} \!+\! \frac{\sigma}{\rho^2}\right) I_s \!+\!
      \begin{bmatrix}
        \changed{0_{s_1\times s_1}}\!\!\! \\ & \left[ \frac{2(1-t)}{l_d}  A_d A_d^\top \right]_{d = 2,\dots, D-1}^{\diag}
      \end{bmatrix}\text.
      % \label{thm:prf12}
\end{align*}
}
\vspace{-1em}
% \ifdefined\SubmittedPaper
% \else
\begin{lemma}[\cite{cao2019dynamic,deng2016global}]
  \label{lem:strong_convexity}
  For  any strongly convex function $f$ with moduli $\alpha > 0$, the following inequality holds for all $z, z^\prime \in \mathrm{Dom}(f)$,    
  $ \alpha \Vert z - z^\prime \Vert_2^2 \leq (\nabla f(z) - \nabla f(z^\prime))^\top 
  (z - z^\prime).$
\end{lemma}
\begin{lemma}[\cite{cao2019dynamic,deng2016global}]\label{lem:convex_lip}
    For any differentiable convex function $f$ having Lipschitz continuous gradient with Lipschitz constant $L$, the following inequality holds for all $z, z^\prime \in \mathrm{Dom}(f)$,
        $\Vert \nabla f(z) - \nabla f(z^\prime) \Vert_2^2 \leq L (z - z^\prime)^\top (\nabla f(z) - \nabla f(z^\prime))$.
\end{lemma}
% Recall \eqref{eq:forward}, $x_d^{k+1}$ satisfies the generalized first-order condition \cite{boyd2003subgradient}:
% \begin{align}
%     \label{thm:prf_gradxk}
%     \begin{cases}
%         \nabla f_0 (x_0^{k+1}) + A_0^\top (y_0^{k+1} - \rho B_0 (x_1^{k+1} - x_1^k)) = 0,
%         \\
%         \nabla f_D (x_D^{k+1}) + B_{D-1}^\top y_{D-1}^{k+1} = 0,
%         \\
%         \nabla f_d (x_d^{k+1}) + A_d^\top (y_d^{k+1} - \rho B_d (x_{d+1}^{k+1} - x_{d+1}^k)) + B_{d-1}^\top y_{d-1}^{k+1} = 0, ~&\forall~0<d<D.
%     \end{cases}
% \end{align}
%
\changed{Stack the four $s$-dimensional block vectors as
\[
\begin{aligned}
\bm\eta^k=\operatorname{col}\big(&
[\tilde y_d^k]_{d=1}^{D-1},[\tilde x_{d+1}^k]_{d=1}^{D-1}, [r_d^k]_{d=1}^{D-1},[\delta_d^k]_{d=1}^{D-1}\big),
\end{aligned}
\]
}
where $([ x_d^{\star}]_{d\in \bbN_{D}}, [ y_d^{\star}]_{d\in \bbN_{D-1}})$ is a KKT solution of \eqref{thm:prf_KKT1}-\eqref{thm:prf_KKT4},
and the residual $r_d^k$ is defined in \eqref{eq:conv-proof:res}.
% and $\bm{\eta}^k$ collects all the variables.

Let $\nabla f_d^k=\nabla f_d(x_d^k)$ and $\nabla f_d^\star=\nabla f_d(x_d^\star)$.
\changed{Lemmas~\ref{lem:convex_lip} and~\ref{lem:strong_convexity} result in the following inequalities}
\begin{align}
  \Vert \nabla f_d^{k+1} \!-\! \nabla f_d^\star \Vert_2^2 
  &\leq l_d(x_d^{k+1} \!-\! x_d^\star)^\top 
  (\nabla f_d^{k+1} \!-\! \nabla f_d^\star)\text,
  \label{eq:prf_rate1}
  \\
  \alpha_d \Vert x_d^{k+1} \!-\! x_d^\star \Vert_2^2 &\leq (\nabla f_d^{k+1} \!-\! \nabla f_d^\star)^\top \!
  (x_d^{k+1} \!-\! x_d^\star)\text.
  \label{eq:prf_rate2}
\end{align}
\changed{where \eqref{eq:prf_rate1} and \eqref{eq:prf_rate2} are for $2\leq d\leq D$ and for $1\leq d\leq D$.}
From \eqref{thm:prf_KKT1}-\eqref{thm:prf_KKT4} and \eqref{thm:prf_KKTyk1}-\eqref{thm:prf_KKTyk3},
% and using the definitions of $\tilde{x}_d^k$ and $\tilde{y}_d^k$,
we have
\begin{align}
    &\nabla f_1^{k+1} - \nabla f_1^\star = - \!A_1^\top \tilde{y}_1^{k+1} +  A_1^\top \delta_1^{k+1},
    \label{eq:prf_rate2-1}
    \\
    &\nabla f_D^{k+1} - \nabla f_D^\star = \!-\! B_{D-1}^\top \tilde{y}_{D-1}^{k+1},
    \label{thm:prf_rate3}
    \\
    &\nabla f_d^{k+1} - \nabla f_d^\star =  \!-\! B_{d-1}^\top  \tilde{y}_{d-1}^{k+1} 
     -\! A_d^\top \tilde{y}_d^{k+1} \!+\!  A_d^\top \delta_d^{k+1}\text,
    \label{eq:prf_rate4}
\end{align}
for $2\leq d\leq D-1$.
Then we can rewrite \eqref{eq:prf_rate1} and \eqref{eq:prf_rate2} with the help of \eqref{eq:prf_rate2-1}-\eqref{eq:prf_rate4} as follows.
At $d = D$,
\begin{align}
  0 &\geq \frac{1}{l_D}\Vert B_{D-1}^\top \tilde{y}_{D-1}^{k+1} \Vert_2^2 
      + (\tilde{x}_D^{k+1} )^\top \tilde{y}_{D-1}^{k+1},
      \label{thm:prf_rate5} \\
  0 &\geq \alpha_D \Vert x_D^{k+1} - x_D^\star \Vert_2^2 + (\tilde{x}_D^{k+1} )^\top \tilde{y}_{D-1}^{k+1} \nonumber \\
    &\geq \frac{\alpha_D}{\Vert B_{D-1} \Vert_2^2} \Vert \tilde{x}_D^{k+1} \Vert_2^2 + (\tilde{x}_D^{k+1} )^\top \tilde{y}_{D-1}^{k+1}\text.
      \label{thm:prf_rate6}
\end{align} 
At $d = 2,\dots, D-1$, noting that $A_d(x_d^{k+1} - x_d^\star) = r_d^{k+1} - B_d(x_{d+1}^{k+1} - x_{d+1}^\star) = r_d^{k+1} - \tilde{x}_{d+1}^{k+1}$, denote $\Gamma_d^{k+1} = A_d^\top \tilde{y}_d^{k+1} + B_{d-1}^\top \tilde{y}_{d-1}^{k+1} - A_d^\top \delta_d^{k+1}$, we have
\begin{align}
  0 &\geq
      \begin{multlined}[t]
        \frac{1}{l_d} \Vert \Gamma_d^{k+1} \Vert_2^2
        + (x_d^{k+1} - x_d^\star)^\top \Gamma_d^{k+1}
      \end{multlined} \nonumber\\
    &=
      \begin{multlined}[t]
        \frac{1}{l_d} \Vert \Gamma_d^{k+1} \Vert_2^2 + (\tilde{x}_d^{k+1})^\top \tilde{y}_{d-1}^{k+1} + \changed{\psi_d^{k+1}}
        % \\
        % + (r_d^{k+1} - \tilde{x}_{d+1}^{k+1})^\top \tilde{y}_d^{k+1} - (x_d^{k+1} - x_d^\star)^\top\delta_{d+1}^{k+1}
      \end{multlined} \nonumber\\
    &\!\underset{\eqref{eq:conv-proof:prf_expr}}{=}
      % \begin{multlined}[t]
        \frac{1}{l_d} \Vert \Gamma_d^{k+1} \Vert_2^2 \!+\! (\tilde{x}_d^{k+1})^\top \tilde{y}_{d-1}^{k+1} \!-\! (\tilde{x}_{d+1}^{k+1})^\top \tilde{y}_d^{k+1} \notag \\
        &\qquad \qquad - \frac{1}{2} (V_d^k \!-\! V_d^{k+1} \!-\! \rho \Vert r_d^{k+1} \!-\! \frac{1}{\rho} \delta_d^{k+1} \Vert_2^2)\text.
      % \end{multlined} 
      \label{eq:prf_rate7}\\
  0 &\geq
      \begin{multlined}[t]
        \alpha_d \Vert x_d^{k+1} - x_d^\star \Vert_2^2 
        + (x_d^{k+1} - x_d^\star)^\top \Gamma_d^{k+1}
      \end{multlined} \nonumber\\
    &\geq
      \begin{multlined}[t]
        \frac{\alpha_d}{\Vert B_{d-1} \Vert_2^2} \Vert \tilde{x}_d^{k+1} \Vert_2^2 
        + (\tilde{x}_d^{k+1})^\top \tilde{y}_{d-1}^{k+1} + \changed{\psi_d^{k+1}} 
        % \\
        % + (r_d^{k+1} - \tilde{x}_{d+1}^{k+1})^\top \tilde{y}_d^{k+1} - (x_d^{k+1} \!-\! x_d^\star)^\top\delta_d^{k+1}
      \end{multlined}
      \nonumber\\
    &\underset{\eqref{eq:conv-proof:prf_expr}}{=} \frac{\alpha_d}{\Vert B_{d-1} \Vert_2^2} \Vert \tilde{x}_d^{k+1} \Vert_2^2
      + (\tilde{x}_d^{k+1})^\top \tilde{y}_{d-1}^{k+1} - (\tilde{x}_{d+1}^{k+1})^\top \tilde{y}_d^{k+1}
      \nonumber\\
    &\phantom{=} \qquad - \frac{1}{2} (V_d^k - V_d^{k+1} - \rho \Vert r_d^{k+1} - \frac{1}{\rho} \delta_d^{k+1} \Vert_2^2)\text, \label{eq:prf_rate8}
\end{align}
where \changed{$\psi_d^{k+1} = (r_d^{k+1} - \tilde{x}_{d+1}^{k+1})^\top \tilde{y}_d^{k+1} - (x_d^{k+1} \!-\! x_d^\star)^\top A_d^\top\delta_d^{k+1}$.}
% $V_d^k$ is defined in \eqref{eq:conv-proof:Vk}.
\changed{At $d=1$, the strong-convexity term gives}
\changed{%
\begin{align}
0&\geq(x_1^{k+1}-x_1^\star)^\top A_1^\top
(\tilde y_1^{k+1}-\delta_1^{k+1})
\notag\\
&= \!- (\tilde x_2^{k+1})^\top\tilde y_1^{k+1}
\!-\! \frac12(V_1^k \!-\! V_1^{k+1}) \!+\! \frac{\rho}{2}\left\|r_1^{k+1} \!-\! \rho^{-1}\delta_1^{k+1}\right\|_2^2,
\label{eq:rate_prf0}
\end{align}
}
Multiplying \eqref{thm:prf_rate6} and \eqref{eq:prf_rate8} by $t\in (0,1)$, and multiplying \eqref{thm:prf_rate5} and \eqref{eq:prf_rate7} by $(1-t)$, 
then adding the results and \eqref{eq:rate_prf0} together, 
we obtain the following inequality
\begin{align*}
  V^k \!-\! V^{k+1} 
  &\geq
    % \begin{multlined}[t]
      \frac{2(1-t)}{l_D}\Vert B_{D-1}^\top \tilde{y}_{D-1}^{k+1} \Vert_2^2 \\
      &+ \!\! \sum_{d=2}^{D-1}\frac{2(1-t)}{l_d} \Vert A_d^\top \tilde{y}_d^{k+1} \!+\! B_{d-1}^\top \tilde{y}_{d-1}^{k+1} \!-\! A_d^\top \delta_d^{k+1} \Vert_2^2 \\
      &+ \!\! \sum_{d=1}^{D-1} \changed{\Big(}\frac{2t \alpha_{d+1}}{\Vert B_d \Vert_2^2} \Vert \tilde{x}_{d+1}^{k+1} \Vert_2^2 + \rho \Vert r_d^{k+1} \!-\! \frac{1}{\rho} \delta_d^{k+1}  \Vert_2^2 \changed{\Big)}\text.
    % \end{multlined}
    % \label{eq:prf_rate9}
\end{align*}
Adding $\sigma\sum_{d=1}^{D-1}(\Vert \tilde{x}_{d+1}^k \Vert_2^2 - \Vert \tilde{x}_{d+1}^{k+1} \Vert_2^2)$ to both sides of the inequality %  \eqref{thm:prf_rate9}
and noting %that
$\Vert \tilde{x}_{d+1}^k \Vert_2^2  = \Vert \tilde{x}_{d+1}^{k+1} - \frac{1}{\rho} \delta_d^{k+1} \Vert_2^2$,
\changed{
we have
\begin{align}
  &\hat V^k - \hat V^{k+1} \geq (\bm\eta^{k+1})^\top \Xi  \bm\eta^{k+1}\text.
    \label{eq:prf_rate12}
\end{align}
Since $(\bm\eta^{k+1})^\top\Theta\bm\eta^{k+1}=\hat V^{k+1}$, \eqref{thm:sdp} and \eqref{eq:prf_rate12} imply \eqref{eq:lya_ctrc}.%
\ifdefined\SubmittedPaper
\else
{} The block-wise identification of $\Xi$ from the inequality above is detailed in Appendix~\ref{prop:proof_rate}.
\fi
}
\changed{%
Lemma~\ref{lem:strong_convexity}, \eqref{thm:prf_KKTyk1}--\eqref{thm:prf_KKTyk3}, and Cauchy--Schwarz inequality give
\(\alpha_1\|x_1^{k+1}-x_1^\star\|_2
\leq\|A_1^\top(\delta_1^{k+1}-\tilde y_1^{k+1})\|_2\).
Using $\delta_1^{k+1}=\rho(\tilde x_2^{k+1}-\tilde x_2^k)$ and Cauchy--Schwarz inequality gives
\begin{align}
   \|x_1^{k+1}-x_1^\star\|_2^2
   &\leq\frac{2\|A_1\|_2^2}{\alpha_1^2}\Big[\|\tilde y_1^{k+1}\|_2^2\notag\\*
   &\qquad\qquad+2\rho^2\big(\|\tilde x_2^{k+1}\|_2^2+\|\tilde x_2^k\|_2^2\big)\Big].
    \label{thm:x1_rate}
\end{align}
From the first inequality of \eqref{thm:prf_rate6}, we obtain
\begin{align}
  \Vert x_D^{k+1} - x_D^\star \Vert_2^2 
  \leq \frac{1}{2\alpha_D} \left(\| \tilde{x}_D^{k+1} \|_2^2 + \| \tilde{y}_{D-1}^{k+1}\|_2^2\right).
    \label{thm:xD_rate}
\end{align}
For $2\leq d\leq D-1$, similarly using Cauchy--Schwarz gives
\begin{align}
  \|x_d^{k+1}-x_d^\star\|_2^2
  &\leq \frac{3}{\alpha_d^2}\Big[\|A_d\|_2^2\|\tilde y_d^{k+1}\|_2^2
  +\|B_{d-1}\|_2^2\|\tilde y_{d-1}^{k+1}\|_2^2 \notag\\
  &\quad+2\rho^2\|A_d\|_2^2
  \big(\|\tilde x_{d+1}^{k+1}\|_2^2+\|\tilde x_{d+1}^k\|_2^2\big)\Big].
  \label{thm:xd_rate}
\end{align}
Define $\beta_1=\frac{2\|A_1\|_2^2}{\alpha_1^2}\max\{1,2\rho^2\}$,
$\beta_D=\frac{1}{2\alpha_D}$, and, for $2\leq d\leq D-1$,
\(\beta_d=\frac{3}{\alpha_d^2}\max\{\|A_d\|_2^2,\|B_{d-1}\|_2^2,2\rho^2\|A_d\|_2^2\}\).
Let $\beta=\max_{1\leq d\leq D}\beta_d$ and $c_\rho=\max\{\rho,\rho^{-1}\}$.
Since $\sum_{d=1}^{D-1}(\|\tilde y_d^k\|_2^2+\|\tilde x_{d+1}^k\|_2^2)\leq c_\rho V^k$ and $V^k\leq\hat V^k$ for $\sigma>0$, summing \eqref{thm:x1_rate}--\eqref{thm:xd_rate} and using \eqref{eq:lya_ctrc} gives
\begin{align}
  \sum_{d=1}^{D}\|x_d^{k+1}-x_d^\star\|_2^2
  &\leq\beta\sum_{d=1}^{D-1}\Big(
  \|\tilde x_{d+1}^{k+1}\|_2^2+2\|\tilde y_d^{k+1}\|_2^2\Big)\notag\\
  &\quad+\beta\|\tilde x_D^{k+1}\|_2^2
  +\beta\sum_{d=1}^{D-1}\|\tilde x_{d+1}^k\|_2^2\notag\\
  &\leq\beta c_\rho\big(2V^{k+1}+V^k\big)\notag\\
  &\leq3\beta c_\rho\left(\frac{1}{1+\gamma^\star}\right)^k\hat V^0.
  \label{eq:primal-rate-squared}
\end{align}
By Cauchy--Schwarz, $(\sum_d\|x_d^{k+1}-x_d^\star\|_2)^2\leq D\sum_d\|x_d^{k+1}-x_d^\star\|_2^2$.
Thus \eqref{eq:primal-rate-squared} proves \eqref{eq:lin-con} with $\calC=\sqrt{3D\beta c_\rho\hat V^0}$, establishing $R$-linear primal convergence.
Strong convexity ensures that the primal minimizer is unique.
}

% \fi % End of detailed proof of Theorem 3

\ifdefined\SubmittedPaper
\else
\section{Proof of Proposition \ref{prop:fullrank}
\label{prop:proof_rate}}

Let $\eta= [y^\top, z^\top, r^\top, \delta^\top ]^\top\in\mathbb R^{4s}$ be arbitrary, with each of $y,z,r,\delta$ partitioned into blocks of sizes $s_1,\ldots,s_{D-1}$ as $y_d $, $z_d$, $r_d$, and $\delta_d \in \bbR^{s_d}$, respectively; the blocks $(y_d,z_d,r_d,\delta_d)$ play the roles of $(\tilde y_d^{k+1},\tilde x_{d+1}^{k+1},r_d^{k+1},\delta_d^{k+1})$ in $\bm\eta^{k+1}$.
Set $a_d=2t\alpha_{d+1}/\|B_d\|_2^2$ for $d\in\bbN_{D-1}$, and $w_d=2(1-t)/l_d$ for $d=2,\ldots,D$.
To derive $\Xi$ from the descent bound preceding \eqref{eq:prf_rate12}, define
\[
g_d=A_d^\top(y_d-\delta_d)+B_{d-1}^\top y_{d-1},
\quad 2\leq d\leq D-1.
\]
Each expands as
\begin{align}
\|g_d\|_2^2
= \; & y_d^\top A_dA_d^\top y_d
+y_{d-1}^\top B_{d-1}B_{d-1}^\top y_{d-1}
\notag\\
&+\delta_d^\top A_dA_d^\top\delta_d
+2y_d^\top A_dB_{d-1}^\top y_{d-1}
\notag\\
&-2\delta_d^\top A_dA_d^\top y_d
-2\delta_d^\top A_dB_{d-1}^\top y_{d-1}\text.
\label{prf:thm_gd}
\end{align}
% For each $d$, the remaining terms, including the correction from $V$ to $\hat V$, expand as
% \begin{align*}
% &a_d\|z_d\|_2^2+\rho\|r_d-\rho^{-1}\delta_d\|_2^2\\
% &\qquad+\sigma\big(\|z_d-\rho^{-1}\delta_d\|_2^2-\|z_d\|_2^2\big)\\
% &=a_d\|z_d\|_2^2+\rho\|r_d\|_2^2
% \!-\! 2r_d^\top\delta_d\\
% &\qquad-\frac{2\sigma}{\rho}\delta_d^\top z_d
% +\left(\frac1\rho+\frac{\sigma}{\rho^2}\right)\|\delta_d\|_2^2.
% \end{align*}
% In a symmetric quadratic form, an off-diagonal block $M$ contributes twice the corresponding bilinear term.
% Thus the gradient expansion gives, for $2\leq d\leq D-1$,
% \begin{align*}
% &[\Xi^{(1,1)}]_{d,d-1} =w_dA_dB_{d-1}^\top,\;
% [\Xi^{(4,1)}]_{d,d} =-w_dA_dA_d^\top,\\
% &[\Xi^{(4,1)}]_{d,d-1} =-w_dA_dB_{d-1}^\top.
% \end{align*}
% Collecting the squared $y$ terms gives the diagonal blocks of $\Xi^{(1,1)}$.
% Expanding the remaining terms of the descent bound together with the correction from $V$ to $\hat V$, namely $a_d\|z_d\|_2^2+\rho\|r_d-\rho^{-1}\delta_d\|_2^2+\sigma(\|z_d-\rho^{-1}\delta_d\|_2^2-\|z_d\|_2^2)$, gives $\Xi^{(2,2)}$, $\Xi^{(3,3)}$, $\Xi^{(4,2)}$, and $\Xi^{(4,3)}$; its squared $\delta$ term and the gradient expansion give $\Xi^{(4,4)}$.
% All other blocks vanish, apart from the symmetric transposes.
Pre- and post-multiplying the four-by-four block matrix $\Xi$ by $\eta$ gives
\begin{align}
\eta^\top\Xi\eta
=\; & y^\top\Xi^{(1,1)}y+z^\top\Xi^{(2,2)}z
+r^\top\Xi^{(3,3)}r
\notag\\
&+\delta^\top\Xi^{(4,4)}\delta
+2\delta^\top\Xi^{(4,1)}y
\notag \\
&+2\delta^\top\Xi^{(4,2)}z
+2\delta^\top\Xi^{(4,3)}r\text.
\label{prf:thm_eta_xi_eta}
\end{align}
% For example, the two terms involving blocks $(1,4)$ and $(4,1)$ satisfy
% $y^\top(\Xi^{(4,1)})^\top\delta+\delta^\top\Xi^{(4,1)}y
% =2\delta^\top\Xi^{(4,1)}y$ because they are equal scalars.
% The same reasoning applies to the other two cross terms.
Substituting the individual blocks $y_d, x_d, r_d$, and $\delta_d$ into  \eqref{prf:thm_eta_xi_eta} yields
\begin{subequations}
\begin{align}
y^\top\Xi^{(1,1)}y
&=\sum_{d=1}^{D-1}w_{d+1}y_d^\top B_dB_d^\top y_d + \sum_{d=2}^{D-1}w_dy_d^\top A_dA_d^\top y_d \notag \\
&\qquad+2\sum_{d=2}^{D-1}w_dy_d^\top A_dB_{d-1}^\top y_{d-1}\text, 
\label{eq:id1}
\\
2 \delta^\top\Xi^{(4,1)}y &=-2\sum_{d=2}^{D-1}w_d\delta_d^\top A_dA_d^\top y_d  \notag
\\
&\qquad - 2\sum_{d=2}^{D-1}w_d\delta_d^\top A_dB_{d-1}^\top y_{d-1}, 
\label{eq:id2} 
\\
\delta^\top\Xi^{(4,4)}\delta
&=\sum_{d=1}^{D-1}\left(\frac1\rho+\frac{\sigma}{\rho^2}\right)\|\delta_d\|_2^2  \notag \\
&\qquad +\sum_{d=2}^{D-1}w_d\delta_d^\top A_dA_d^\top\delta_d,
\label{eq:id3}
\\
z^\top\Xi^{(2,2)}z &=\sum_{d=1}^{D-1}a_d\|z_d\|_2^2, 
\label{eq:id4}
\\
r^\top\Xi^{(3,3)}r &=\rho\sum_{d=1}^{D-1}\|r_d\|_2^2,
\label{eq:id5}
\\
2\delta^\top\Xi^{(4,2)}z &=-\frac{2\sigma}{\rho}\sum_{d=1}^{D-1}\delta_d^\top z_d,
\label{eq:id6}
\\
2\delta^\top\Xi^{(4,3)}r &=-2\sum_{d=1}^{D-1}\delta_d^\top r_d.
\label{eq:id7}
\end{align}
\end{subequations}
For each $d$, combining the first term of \eqref{eq:id3} with the terms of \eqref{eq:id4}--\eqref{eq:id7} gives
\begin{align*}
&a_d\|z_d\|_2^2 \!+\! \rho\|r_d\|_2^2-2\delta_d^\top r_d \!-\! \frac{2\sigma}{\rho}\delta_d^\top z_d
\!+\! \left(\frac1\rho \!+\! \frac{\sigma}{\rho^2}\right)\|\delta_d\|_2^2\\
&=(a_d-\sigma)\|z_d\|_2^2
+\sigma\|z_d \!-\! \rho^{-1}\delta_d\|_2^2+\rho\|r_d \!-\! \rho^{-1}\delta_d\|_2^2.
\end{align*}
Reindexing the first sum in $y^\top\Xi^{(1,1)}y$ separates the leaf term
$\sum_{d=1}^{D-1}w_{d+1}y_d^\top B_dB_d^\top y_d =w_D\|B_{D-1}^\top y_{D-1}\|_2^2
+\sum_{d=2}^{D-1}w_d\|B_{d-1}^\top y_{d-1}\|_2^2$.
By \eqref{prf:thm_gd}, all terms containing $A_d$ and $B_{d-1}$ in \eqref{eq:id1}--\eqref{eq:id3}, excluding the leaf term $w_D\|B_{D-1}^\top y_{D-1}\|_2^2$, then combine into $\sum_{d=2}^{D-1}w_d\|g_d\|_2^2$.
Combining the two equalities above with \eqref{prf:thm_gd} and \eqref{eq:id1}--\eqref{eq:id7}, the expansion \eqref{prf:thm_eta_xi_eta} becomes
\begin{align}
&\eta^\top\Xi\eta
= w_D 
\|B_{D-1}^\top y_{D-1}\|_2^2 + \sum_{d=2}^{D-1} w_d \| g_d \|_2^2 + \notag\\
&\sum_{d=1}^{D-1} \!\! \Big[(a_d \!-\! \sigma)\|z_d\|_2^2
+\sigma\|z_d \!-\! \frac{\delta_d}{\rho}\|_2^2 +\rho \|r_d \!-\! \frac{\delta_d}{\rho}\|_2^2\Big].
\label{eq:xi-sos}
\end{align}
All coefficients are positive because $t\in(0,1)$, $l_d>0$, $\rho>0$, and $0<\sigma<\min_d a_d$.
Hence $\eta^\top\Xi\eta\geq0$ for every $\eta$, so $\Xi\succeq0$, independently of whether $\eta$ is generated by the algorithm.

If $\eta^\top\Xi\eta=0$, the last three terms of the sum in \eqref{eq:xi-sos} force $z_d=\delta_d=r_d=0$ for every $d$.
The remaining squared terms then give
\[
\begin{aligned}
B_{D-1}^\top y_{D-1}&=0,\\
A_d^\top y_d+B_{d-1}^\top y_{d-1}&=0,\quad(2\leq d\leq D-1).
\end{aligned}
\]
Full row rank of $\Lambda$ in \eqref{prop:matrix} implies $y_d=0$ for all $1\leq d\leq D-1$. Thus, $\Xi$ is a positive definite matrix.
% Recursing backward over $d=D-1,\ldots,2$ and using the full row rank of each $B_{d-1}$ gives $y_{d-1}=0$, so $y=0$.
% Thus $\eta=0$ is the only zero of the quadratic form, proving $\Xi\succ0$.
% For $D=2$, the intermediate sum is empty and $B_1^\top y_1=0$ gives the same conclusion.

Finally, $\Theta\succeq0$ and $\|\Theta\|_2=\max\{\rho^{-1},\rho+\sigma\}>0$.
Choose $\gamma=\lambda_{\min}(\Xi)/(2\|\Theta\|_2)>0$.
Then
\[
\begin{aligned}
\Xi-\gamma\Theta&\succeq
\bigl(\lambda_{\min}(\Xi)-\gamma\|\Theta\|_2\bigr)I_{4s}\\
&=\tfrac12\lambda_{\min}(\Xi)I_{4s}\succ0,
\end{aligned}
\]
which proves feasibility of \eqref{thm:sdp}.

\fi

%%% Local Variables:
%%% mode: LaTeX
%%% TeX-master: "main"
%%% End: